\documentclass[twoside]{article}
\usepackage[accepted]{aistats2026}

\usepackage{xcolor}
\usepackage{amsfonts}
\usepackage{enumitem}
\usepackage{amsthm}
\usepackage{stmaryrd}
\usepackage{hyperref}
\usepackage{natbib}
\usepackage{multirow}

\newtheorem{theorem}{Theorem}[section]
\newtheorem{lemma}[theorem]{Lemma}
\newtheorem{corollary}[theorem]{Corollary}
\newtheorem{proposition}[theorem]{Proposition}

\theoremstyle{definition}
\newtheorem{definition}[theorem]{Definition}

\newtheorem{assumption}{Assumption}

\theoremstyle{remark}
\newtheorem{remark}[theorem]{Remark}

\begin{document}

\twocolumn[

\aistatstitle{On propagation of chaos for the Fisher-Rao gradient flow in entropic mean-field optimization}

\aistatsauthor{ Petra Lazi\'{c} \And Linshan Liu \And  Mateusz B.\ Majka }

\aistatsaddress{ University of Ljubljana and\\ University of Zagreb  \And  Heriot-Watt University and\\ Maxwell Institute for\\ Mathematical Sciences  \And Heriot-Watt University and\\ Maxwell Institute for\\ Mathematical Sciences} ]

\begin{abstract}
We consider a class of optimization problems on the space of probability measures motivated by the mean-field approach to studying neural networks. Such problems can be solved by constructing continuous-time gradient flows that converge to the minimizer of the energy function under consideration, and then implementing discrete-time algorithms that approximate the flow. In this work, we focus on the Fisher-Rao gradient flow and we construct an interacting particle system that approximates the flow as its mean-field limit. We discuss the connection between the energy function, the gradient flow and the particle system and explain different approaches to smoothing out the energy function with an appropriate kernel in a way that allows for the particle system to be well-defined. We provide a rigorous proof of the existence and uniqueness of thus obtained kernelized flows, as well as a propagation of chaos result that provides a theoretical justification for using the corresponding kernelized particle systems as approximation algorithms in entropic mean-field optimization.
\end{abstract}

\section{INTRODUCTION}

We consider the following optimization problem on the space of probability measures $\mathcal{P}(\mathcal{X})$ on $\mathcal{X} \subset \mathbb{R}^d$
\begin{equation}\label{eq: energy}
\min_{m \in \mathcal{P}(\mathcal{X})}V^\sigma(m), \quad  V^\sigma(m):= F(m) + \sigma \operatorname{KL}(m|\pi),
\end{equation}
where $F: \mathcal{P}(\mathcal{X}) \to \mathbb{R}$ is a (possibly non-convex) functional bounded from below, $\sigma > 0$ is a regularization parameter, $\pi \in \mathcal{P}(\mathcal{X})$ is a fixed reference measure and $\operatorname{KL}$ denotes the relative entropy (the KL-divergence). While some general results in Section \ref{sec: main} will be stated for domains $\mathcal{X} \subset \mathbb{R}^d$ which do not have to be compact, for the crucial examples studied in Section \ref{sec: verification} we will additionally require $\mathcal{X}$ to be compact. In recent years, there has been considerable interest in problems of this type, motivated by the mean-field approach to the problem of training neural networks (see \cite{Mei_2018} or \cite[Section 3]{hu2021mean} and the references therein), as well as in the context of reinforcement learning, in policy optimization for entropy-regularized Markov Decision Processes with neural network approximation in the mean-field regime \citep{leahy, LascuMajka2025}.

In order to solve \eqref{eq: energy}, one typically aims to construct a gradient flow $(\mu_t)_{t \geq 0}$ on $\mathcal{P}(\mathcal{X})$ that converges to the minimizer $m^{*,\sigma}$ of \eqref{eq: energy} when $t \to \infty$. The most commonly used example is the Wasserstein gradient flow
\begin{equation}\label{eq: Wasserstein_gradient_flow}
\partial_t \mu_t =  \nabla \cdot\left( \mu_t \nabla \frac{\delta V^\sigma}{\delta \mu}(\mu_t, \cdot)\right)
\end{equation}
defined via the flat derivative (first variation) $\frac{\delta V^\sigma}{\delta \mu}$ of the energy function $V^\sigma$ (see Definition \ref{def fderivative}). The conditions guaranteeing the convergence of \eqref{eq: Wasserstein_gradient_flow} to $m^{*,\sigma}$ have been studied by numerous authors in various settings, including \cite{AmbrosioGigliSavare2008, hu2021mean, Nitanda2022, Chizat2022, leahy} and many others. 

However, from the point of view of applications, an equally important question is how to approximate gradient flows such as \eqref{eq: Wasserstein_gradient_flow} by a practically implementable algorithm. One possible approach leads through the Jordan-Kinderlehrer-Otto (JKO) schemes (see \cite{JKO} for the original paper or \cite{korbaproximal, lascu2024linearconvergenceproximaldescent} for more recent expositions). Another, which we are going to focus on in the present paper, utilizes an interpretation of \eqref{eq: Wasserstein_gradient_flow} as the mean-field limit of an interacting particle system. In the latter approach, one typically aims to prove a propagation of chaos result, i.e., a result that shows that as the number of particles approaches infinity, particles become independent and they all follow the same mean-field dynamics. This can be then used as a theoretical justification that an appropriately constructed interacting particle system may be used to produce (after a discretisation) an algorithm that approximates the minimizer of \eqref{eq: energy} when the number of particles and the number of iterations are both sufficiently large.

Propagation of chaos for the Wasserstein gradient flow has been studied in detail in various settings (utilizing the interpretation of \eqref{eq: Wasserstein_gradient_flow} as the Fokker-Planck equation for the mean-field Langevin SDE). A (far from complete) list of references includes \cite{ChenRenWang, MonmarcheRenWang2024, durmusandreasgullinzimmer2020, carrilloMcCannVillani2003, malrieu2001,  delaruetse2021, lackerleflem2023, SuzukiNitandaWu2023, Nitanda2024, NitandaLee2025, GuKim2025}. A related active strain of research involves the propagation of chaos for kinetic models \cite{monmarche2017, guillinmonmarche2021, schuh2022, ChenLinRenWang2024}.

In the present work, we focus on a different gradient flow, the so-called Fisher-Rao gradient flow given by
\begin{equation}\label{eq: FR}
\partial_t \mu_t = -\mu_t \frac{\delta V^\sigma}{\delta \mu}(\mu_t,\cdot).
\end{equation}
The interest in studying this flow in the context of  optimization problems \eqref{eq: energy} is motivated by the fact that in some settings, its convergence to the minimizer could be easier to verify than for the Wasserstein flow \cite{Liu-Majka-Szpruch-2023, Kerimkulov2025, lascu2024a}. There has been a considerable literature studying the fundamental properties of Fisher-Rao gradient flows such as well-posedness in various settings, see e.g.\ \cite{carrillo2024, zhu2024kernelapproximationfisherraogradient} and the references therein, also in combination with the Wasserstein flow as the Wasserstein-Fisher-Rao gradient flow (also known as Hellinger-Kantorovich) \cite{LieroMielkeSavare2018, Gallouet2017, LuLuNolen2019, vandenEijnden2019}.

However, unlike for the Wasserstein gradient flow \eqref{eq: Wasserstein_gradient_flow}, in the context of the Fisher-Rao flow \eqref{eq: FR} much less is known about particle approximations (with a few exceptions that will be discussed in detail in Section \ref{subsec: other_poc}). In particular, the main goal of the present paper is to fill a gap in the literature by providing a rigorous proof for a particle approximation for the Fisher-Rao gradient flow \eqref{eq: FR} corresponding to a class of minimization problems of the form \eqref{eq: energy}.

\subsection{Contributions}
We summarize the main contributions of our paper:
\begin{itemize}
	\item In Section \ref{sec: main}, we propose a general framework for studying particle approximations of Fisher-Rao gradient flows. This part extends the results from \cite{LeCavil2017, LuLuNolen2019, vandenEijnden2019, Domingo-EnrichBruna} (see Remark \ref{remark: Domingo} and Section \ref{subsec: other_poc} for details).
	
	\item In Section \ref{sec: verification}, we discuss different approaches to smoothing out Fisher-Rao flows (which is crucial for practical implementation) and we show that all the discussed methods fall within our framework from Section \ref{sec: main}. This part expands on \cite{LuLuNolen2019, Pampel_2023, Lu_2023, Carrillo2019}.
	
	\item In Section \ref{sec: particle_system}, we propose a practical algorithm for approximating Fisher-Rao flows, utilizing a method from Section \ref{sec: verification}.
\end{itemize}

\section{MAIN RESULTS}\label{sec: main}

In this section, we work on a general (not necessarily compact) space $\mathcal{X} \subset \mathbb{R}^d$ and we focus on the flow
\begin{equation}\label{eq: FR with a}
    \partial_t \mu_t = - \mu_t a(\mu_t, \cdot)
\end{equation}
where the flat derivative $\frac{\delta V^{\sigma}}{\delta \mu}$ on the right hand side of \eqref{eq: FR} has been replaced with a function $a: \mathcal{P}(\mathcal{X}) \times \mathcal{X} \to \mathbb{R}$. In Section \ref{sec: verification}, we will explain the rationale behind approximating the flat derivative $\frac{\delta V^{\sigma}}{\delta \mu}$ of the energy function $V^{\sigma}$ in \eqref{eq: energy} with its kernelized counterpart. Different choices of kernelizations will lead to different choices of $a$ and hence in this section we keep our notation general, in order to produce a broadly applicable theoretical framework. For convenience, we will refer to flows \eqref{eq: FR with a} as Fisher-Rao flows, although it should be understood that they are "true" Fisher-Rao flows only for some choices of $a$ (see Section \ref{sec: kernelization} for more details).

We remark that we always choose $a$ in a way that ensures that $\mu_t$ remains a probability distribution. This is achieved by requiring that for all $\mu \in \mathcal{P}(\mathcal{X})$,
\begin{equation}\label{eq: mass preservation}
\int_{\mathcal{X}} a(\mu,x) \mu(dx) = 0,
\end{equation}
which then immediately implies that $\partial_t \int_{\mathcal{X}} \mu_t(dx) = - \int_{\mathcal{X}} a(\mu_t,x) \mu_t(dx) =0$, and hence the flow preserves the mass of the initial measure. It is easy to check that all examples of $a$ studied in Section \ref{sec: verification} satisfy property \eqref{eq: mass preservation}.

\subsection{Preliminaries}

Before we proceed, let us introduce some necessary notation. 
For \(p\in[1,\infty)\), and for any normed vector space $(E,\lVert \cdot \rVert)$, we define the set of probability measures with finite \(p\)-th moment \(\mathcal{P}_p(E)\) as
$$\mathcal{P}_p(E) := \left\{ \mu \in \mathcal{P}(E) : \int_{E} \lVert x\rVert^p \mu(dx) < \infty \right\}.$$
We equip \(\mathcal{P}_p(E)\) with the \(p\)-Wasserstein distance
\[
\mathcal{W}_p(m,m')
:= \Bigg(\inf_{\pi\in\Pi(m,m')}
\int_E \lVert x-y \rVert^p\, \pi(dx,dy)\Bigg)^{1/p},
\]
where \(\Pi(m,m')\subset \mathcal{P}(E \times E)\) is the set of all couplings of \(m\) and \(m'\).

For any $T>0$, we consider the path space \(\mathcal{C}([0,T];E)\) with the supremum norm
\[
\lVert x-y \rVert_{T}:=\sup_{0\le s\le T} \lVert x_s-y_s \rVert,
\]
and write \(\mathcal{P}_p(\mathcal{C}([0,T];E))\) for the corresponding \(p\)-moment space. 
The induced \(p\)-Wasserstein distance on path measures is
\[
\mathcal{W}_{p,T}(m,m')
:= \Bigg(\inf_{\pi\in\Pi(m,m')}
\int \lVert x-y \rVert_{T}^{\,p}\, \pi(dx,dy)\Bigg)^{1/p}.
\]

We require $a$ to satisfy the following assumption:
\begin{assumption}\label{assumption1}
    Function $a$ is bounded and Lipschitz, i.e., there exist constants $M_a$, $L_a > 0$ such that for all \(x,y\in\mathcal{X}\) and \(m,m'\in\mathcal{P}_2(\mathcal{X})\),
\begin{equation}\label{eq: boundness_a}
    |a(m,x)| \le M_a,
\end{equation}
and
\begin{equation}\label{eq: Lipschitz_a}
    |a(m,x)-a(m',y)| \le L_a\big(|x-y|+\mathcal{W}_2(m,m')\big).
\end{equation}
\end{assumption}

Following the ideas from \cite{LieroMielkeSavare2018} (see also \cite{Domingo-EnrichBruna}), we work with the notions of lifts and projections of measures. This will allow us to use measures defined on the extended space \(\mathcal{X} \times \mathbb{R}_+\), with the second component representing a local weight.

\begin{definition}[Lifted measure and projection]\label{def: lifted_projection}
Let $\mu \in \mathcal{P}_1(\mathcal{X})$ and $\nu \in \mathcal{P}_1(\mathcal{X} \times \mathbb{R}_+)$.
We say that $\nu$ is a lifted measure of $\mu$, and conversely that $\mu$ is the projection of $\nu$, if for all $\varphi \in \mathcal{C}_c^\infty(\mathcal{X})$,
\begin{equation}\label{eq: projection}
    \int_{\mathcal{X}} \varphi(x) \, d\mu(x)
    = \int_{\mathcal{X} \times \mathbb{R}_+} w \, \varphi(x) \, d\nu(x, w),
\end{equation}
where $\mathcal{C}_c^\infty(\mathcal{X})$ denotes the space of smooth, compactly supported functions on $\mathcal{X}$.  
In this case we define the projection operator
\[
    \texttt{h} : \mathcal{P}_1(\mathcal{X} \times \mathbb{R}_+) \to \mathcal{P}_1(\mathcal{X}),
\]
so that $\mu = \texttt{h} \nu$ whenever \eqref{eq: projection} holds.
\end{definition}

Note that the requirement for the measures in Definition \ref{def: lifted_projection} to have finite first moments ensures that the integral on the right hand side of \eqref{eq: projection} is finite.

\subsection{Existence of mean-field dynamics corresponding to Fisher-Rao flow}

Building on the notion of lifted and projected measures defined above, we now establish a rigorous connection between the Fisher--Rao gradient flow and a corresponding mean-field (single particle) dynamics. 
The key idea is to lift the flow from the probability space $\mathcal{P}_1(\mathcal{X})$ to the extended space $\mathcal{P}_1(\mathcal{X}\times\mathbb{R}_+)$, where the additional coordinate $w$ represents a local mass weight. 
This lifted formulation reveals that the original Fisher--Rao flow can be interpreted as the projection of an evolution equation on the extended space. 
Based on this representation, we construct a mean-field dynamic whose law evolves according to the lifted flow, and show that, under suitable regularity assumptions, this dynamic admits a unique solution on any finite time horizon. We begin by stating the precise correspondence between the lifted and projected flows.

\begin{definition}\label{def: lifted_flow_weak}
Let $\nu_0 \in \mathcal{P}_1(\mathcal{X} \times \mathbb{R}_+)$. We say that $(\nu_t)_{t \ge 0} \subset \mathcal{P}_1(\mathcal{X} \times \mathbb{R}_+)$ is a weak solution to the lifted flow
\begin{equation}\label{eq: lifted_flow}
    \partial_t \nu_t(x,w) = \frac{\partial}{\partial w} \big( \nu_t(x,w) \, w \, a(\texttt{h}\nu_t, x) \big),
\end{equation}
with initial condition $\nu_0$, if for any $\psi \in \mathcal{C}_c^\infty(\mathcal{X} \times \mathbb{R}_+)$ and any $t > 0$ we have
\begin{equation}\label{eq: lifted_weak_form}
\begin{split}
    &\frac{d}{dt} \int_{\mathcal{X} \times \mathbb{R}_+} \psi(x,w) \, d\nu_t(x,w)\\
    &= - \int_{\mathcal{X} \times \mathbb{R}_+} w \, a(\texttt{h}\nu_t, x) \, \frac{\partial}{\partial w} \psi(x,w) \, d\nu_t(x,w).
\end{split}
\end{equation}
\end{definition}

\begin{definition}\label{def: FR_flow_weak}
Let $\mu_0 \in \mathcal{P}_1(\mathcal{X})$.
We say that $(\mu_t)_{t \ge 0} \subset \mathcal{P}_1(\mathcal{X})$ is a weak solution to the Fisher-Rao gradient flow
\begin{equation}\label{eq: FR_flow}
    \partial_t \mu_t = -\mu_t \, a(\mu_t, \cdot),
\end{equation}
with initial condition $\mu_0$, if for any $\psi \in \mathcal{C}_c^\infty(\mathcal{X})$ and any $t>0$ we have
\begin{equation}\label{eq: FR_weak_form}
    \frac{d}{dt} \int_{\mathcal{X}} \psi(x) \, d\mu_t(x)
    = - \int_{\mathcal{X}} \psi(x) \, a(\mu_t, x) \, d\mu_t(x).
\end{equation}
\end{definition}

\begin{remark}\label{remark: cc infinity dense in Lp}
Note that by a standard approximation argument, if $(\nu_t)_{t \ge 0}$ is a weak solution to the lifted flow in the sense of Definition \ref{def: lifted_flow_weak}, then \eqref{eq: lifted_weak_form} holds also for functions of the form
\[
    \psi(x,w) = w \, \phi(x), \quad \phi \in \mathcal{C}_c^\infty(\mathcal{X}).
\]
This will be important for arguments where we switch between the lifted and projected flows, especially in the proof of the following lemma.
\end{remark}

\begin{lemma}\label{lemma: projected_lifted}
Let $(\nu_t)_{t \ge 0} \subset \mathcal{P}_1(\mathcal{X} \times \mathbb{R}_+)$ be a weak solution to the lifted flow \eqref{eq: lifted_flow} in the sense of Definition~\ref{def: lifted_flow_weak}, with initial condition $\nu_0 \in \mathcal{P}_1(\mathcal{X} \times \mathbb{R}_+)$.
Then the projected flow $(\mu_t)_{t \ge 0} := (\texttt{h} \nu_t)_{t \ge 0} \subset \mathcal{P}_1(\mathcal{X})$ is a weak solution to the Fisher--Rao flow \eqref{eq: FR_flow} in the sense of Definition~\ref{def: FR_flow_weak}, with initial condition $\mu_0 := \texttt{h} \nu_0$.
\end{lemma}

With this relation at hand, we can construct a mean-field particle dynamic whose law evolves according to the lifted flow. The precise statement is given in the following theorem.

\begin{theorem}\label{thm: particle_representation}
Let $\nu_0 \in \mathcal{P}_1(\mathcal{X} \times \mathbb{R}_+)$. 
Consider the mean-field system
\begin{equation}\label{eq: mean-field dynamic}
    \begin{split}
        &d X_t = 0, \\
        &d w_t = - w_t \, a\big( \texttt{h} \nu_t, X_t \big) \, dt, \\
        &(X_0, w_0) \sim \nu_0,
    \end{split}
\end{equation}
where $\nu_t := \operatorname{Law}(X_t, w_t)$ denotes the joint law of $(X_t, w_t)$. 
If system \eqref{eq: mean-field dynamic} admits a solution, then the curve $(\nu_t)_{t \ge 0}$ is a weak solution to the lifted Fisher--Rao flow \eqref{eq: lifted_flow} in the sense of Definition~\ref{def: lifted_flow_weak}, and hence $(\mu_t)_{t \ge 0} := (\texttt{h} \nu_t)_{t \ge 0}$ is a weak solution to the Fisher--Rao flow \eqref{eq: FR_flow} in the sense of Definition~\ref{def: FR_flow_weak}.
\end{theorem}

\begin{remark}\label{remark: WFR particle}
 System \eqref{eq: mean-field dynamic} can be interpreted as a particle $X_t$ with a corresponding weight $w_t$. Note that the spatial position \( X_t \) of the particle is sampled once from the initial probability distribution and remains fixed over time, while only the associated weight \( w_t \) evolves according to the dynamic in \eqref{eq: mean-field dynamic}. This reflects the geometry of the Fisher--Rao metric, which governs mass change without spatial transport.
\end{remark}

\begin{remark}\label{remark: Domingo}
    Consider the setting corresponding to our primary object of interest, i.e., when the function $a$ in \eqref{eq: FR with a} is given by the flat derivative of $V^{\sigma}$ defined in \eqref{eq: energy}. Then, for measures $\mu$ such that $V^{\sigma}$ has a flat derivative at $\mu$,
    \begin{equation}\label{eq: a without kernel}
        a(\mu, \cdot) = \frac{\delta F}{\delta \mu}(\mu,\cdot) + \sigma \log \frac{\mu}{\pi}
    \end{equation}
    (up to an additive constant - the details will be discussed in Section \ref{sec: verification}). If we assume the necessary differentiability, we can consider the Wasserstein-Fisher-Rao gradient flow given by
    \begin{equation}\label{eq: WFR}
          \partial_t \mu_t(x) = \nabla \cdot\left( \mu_t \nabla a(\mu_t, x)\right) - \mu_t(x) a(\mu_t, x).
    \end{equation}
    Then, due to the interpretation of the Wasserstein gradient flow as the mean-field Langevin SDE \citep{hu2021mean}, we would expect to obtain the following corresponding mean-field (single particle) system
    \begin{equation}\label{eq: WFRsystem}
    \begin{split}
        d X_t &= - \nabla \left(\frac{\delta F}{\delta \mu}(\texttt{h} \nu_t , X_t) - \sigma \log \pi(X_t)\right)dt \\
        &+ \sqrt{2\sigma}\, dW_t,\\
        d w_t &= -  w_t\, a( \texttt{h} \nu_t, X_t)\, dt,\\
        \nu_t &:= \text{Law}(X_t, w_t),
    \end{split}
\end{equation}
where both the location of the particle and the weight evolve in time. There are, however, several technical challenges with studying system \eqref{eq: WFRsystem}. If \eqref{eq: WFRsystem} is considered on a compact space $\mathcal{X} \subset \mathbb{R}^d$, this creates an issue with the Wasserstein part of the flow (due to the diffusive behaviour of $X_t$) and one needs to make sure that the particle stays on the right space. On the other hand, if $\mathcal{X} = \mathbb{R}^d$, then the Fisher-Rao part becomes problematic since the function $a$ defined in \eqref{eq: a without kernel} is unbounded (due to the flat derivative of the relative entropy being unbounded). While working on the present paper, we were unable to overcome these technical challenges which is why we focus on the system \eqref{eq: mean-field dynamic} corresponding to the "pure" Fisher-Rao gradient flow (and in Section \ref{sec: verification} when we discuss specific choices of $a$ related to \eqref{eq: a without kernel}, we work on a compact $\mathcal{X}$ to ensure that $a$ stays bounded). This is consistent with a recent paper \citep{Domingo-EnrichBruna}, which studied (in the context of two-player game theory) Wasserstein-Fisher-Rao flows corresponding to energy functions \eqref{eq: energy} without entropy-regularization (i.e., with $\sigma =0$), which leads to a system of the form \eqref{eq: WFRsystem} where $X_t$ moves only according to a deterministic gradient descent; as well as Wasserstein flows corresponding to \eqref{eq: energy} with entropy regularization but without the Fisher-Rao part. Similarly to \cite{Domingo-EnrichBruna}, we are unable to cover Wasserstein-Fisher-Rao (WFR) flows corresponding to  \eqref{eq: energy} with entropy-regularization, but unlike \cite{Domingo-EnrichBruna}, we study Fisher-Rao flows for \eqref{eq: energy} with entropy-regularization, which were not covered there. We will attempt to treat the challenging case of WFR flows in a future work.
\end{remark}

We now address the well-posedness of mean-field particle dynamics. 

\begin{theorem}\label{thm: exist_unique_solution}
Suppose that $a$ satisfies Assumption \ref{assumption1}. 
Let the initial law \(\nu_0 \in \mathcal{P}_2(\mathcal{X}\times\mathbb{R}_+)\). Then, for any \(T>0\), there exists a unique random dynamic
\((X_t,w_t)_{t\in[0,T]}\) that solves system \eqref{eq: mean-field dynamic}. In particular,
\[
\operatorname{Law}\big((X_t,w_t)_{t\in[0,T]}\big)\in \mathcal{P}_2\!\big(\mathcal{C}([0,T];\mathcal{X}\times\mathbb{R}_+)\big).
\]
Moreover, the solution is pathwise unique: if two solutions share the same initial condition almost surely, then they coincide for all \(t\in[0,T]\) almost surely.
\end{theorem}

\subsection{Interacting Particle Systems and Propagation of Chaos}

We investigate the approximation of the Fisher--Rao gradient flow by interacting particle systems. 
Our objective is to establish propagation of chaos: as the number of particles \(N \to \infty\), the empirical measure of the interacting system converges to the law of i.i.d.\ copies of the mean-field solution, with convergence measured in the 2-Wasserstein distance.

Before we discuss the interacting particle system that is our main object of interest, in order to make the notation precise we first introduce a system \((X^{i,N}, w^{i,N}_t)\) of $N$ i.i.d. copies of the mean-field dynamic \eqref{eq: mean-field dynamic}, which will be used as an auxiliary tool in our approximation estimates.
 For each \(i \in \llbracket 1, N \rrbracket\), we define the dynamic of \((X^{i,N}, w^{i,N}_t)\) by
\begin{equation}\label{eq: systerm iid}
    \begin{split} 
        & X^{i,N} \sim \mu_0^{i,N}, \\
        & w_0^{i,N} \geq 0, \quad \text{with } \sum_{i=1}^N w^{i,N}_0 = N,\\
        & d w_t^{i,N} = - w_t^{i,N}  a(\texttt{h}\nu_t^i, X^{i,N}) \, dt,
    \end{split}
\end{equation}
where \(\nu_t^i := \operatorname{Law}(X^{i,N}, w_t^{i,N})\) is the marginal law of the \(i\)-th particle. We denote the empirical measure of thus obtained \emph{non-interacting} system by
\begin{equation}\label{eq: empirical_non_interacting}
\nu_t^N := \frac{1}{N} \sum_{i=1}^N \delta_{\left(X^{i,N}, w_t^{i,N}\right)}.
\end{equation}
We will use this notation throughout our proofs in Appendix \ref{sec: POC}.

\noindent\textbf{Interacting particle system.}  
We are now ready to introduce the interacting particle system, where the interaction is governed by weighted empirical distributions. For each \(i \in \llbracket 1, N \rrbracket\), we define \((\tilde{X}^{i,N}, \tilde{w}^{i,N}_t)\) by
\begin{equation}\label{eq: FR_interacting_system}
    \begin{split}
        & \tilde{X}^{i,N} \sim \tilde{\mu}_0^{i,N}, \\
        & \tilde{w}_0^{i,N} \geq 0, \quad \text{with } \sum_{i=1}^N \tilde{w}^{i,N}_0 = N,\\
        & d \tilde{w}_t^{i,N} = - \tilde{w}_t^{i,N} \, a\left(\tilde{\mu}_t^N, \tilde{X}^{i,N}\right) \, dt,
    \end{split}
\end{equation}
where \(\tilde{\mu}_t^N := \frac{1}{N} \sum_{i=1}^N \tilde{w}_t^{i,N} \delta_{\tilde{X}^{i,N}}\) is the weighted empirical distribution. Note that the empirical measure on the extended space is
\begin{equation}\label{eq: empirical_interacting}
\tilde{\nu}_t^N := \frac{1}{N} \sum_{i=1}^N \delta_{\left(\tilde{X}^{i,N}, \tilde{w}_t^{i,N}\right)}, \quad \text{so that} \quad \tilde{\mu}_t^N = \texttt{h} \tilde{\nu}_t^N.
\end{equation}

 We now state the main result on the propagation of chaos.

\begin{theorem}\label{thm: POC}
Suppose that the function $a$ satisfies Assumption \ref{assumption1}. Fix a finite time horizon $T>0$ and let $\nu \in \mathcal{P}_2\!\big(\mathcal{C}([0,T];\mathcal{X}\times\mathbb{R}_+)\big)$ be the law of the single particle $(X,w_t) = (X,w_t)_{t \in [0,T]}$ defined by the mean-field dynamic \eqref{eq: mean-field dynamic}. Choose particles \((X^{i,N}, w^{i,N}_t)\) as i.i.d.\ copies of $(X,w_t)$ as defined in \eqref{eq: systerm iid} and let \((\tilde{X}^{i,N}, \tilde{w}^{i,N}_t)\) be the interacting particle system defined by \eqref{eq: FR_interacting_system}. Suppose further that the initial conditions satisfy
\begin{equation} \label{hypo: initial_condition}
    \lim_{N\to \infty} \frac{1}{N} \sum_{i=1}^N\mathbb{E}\left[ \left|X^{i,N}-\tilde{X}^{i,N} \right|^2+ \left| w^i_0 -\tilde{w}^{i,N}_0 \right|^2 \right] = 0.
\end{equation}
Let $\tilde{\nu}^N \in \mathcal{P}_2\!\big(\mathcal{C}([0,T];\mathcal{X}\times\mathbb{R}_+)\big)$ be the empirical measure of the interacting particle system defined by \eqref{eq: empirical_interacting}. 
Then, 
\begin{equation*}
    \lim_{N\to \infty} \mathbb{E}\left[\mathcal{W}_{2,T}(\tilde{\nu}^N, \nu)\right] = 0.
\end{equation*}
\end{theorem}

We remark that condition \eqref{hypo: initial_condition} is automatically satisfied when the initializations for both systems considered in Theorem \ref{thm: POC} are identical.

\begin{corollary}\label{cor: POC}
    Under the assumptions of Theorem~\ref{thm: POC}, the projected empirical distribution \(\tilde{\mu}_t^N := \texttt{h} \tilde{\nu}_t^N\) converges in the 2-Wasserstein distance, uniformly on $[0,T]$, to the mean-field law \(\mu_t := \texttt{h} \nu_t\). That is, for any finite time horizon \(T > 0\), we have

\begin{equation*}
     \lim_{N\to \infty} \sup_{t\in[0,T]}\mathbb{E}[\mathcal{W}_2\left(\mu_t, \tilde{\mu}_t^N\right)] = 0.
\end{equation*}
\end{corollary}

We remark that Theorem \ref{thm: POC} and Corollary \ref{cor: POC} are non-quantitative, and obtaining convergence rates would require further work, which is beyond the scope of the present paper.

\subsection{Discussion of other propagation of chaos results in the literature}\label{subsec: other_poc}

The framework in \citet{LeCavil2017} establishes propagation of chaos results for particle approximations of the following class of non-conservative nonlinear PDEs:
\begin{equation}\label{eq: Cavil_PDE}
\begin{split}
&\partial_t v = \sum_{i,j=1}^d \partial^2_{ij}\left( (\Phi\Phi^\top)_{i,j}(t,x,v) v \right)\\
&- \nabla \cdot (g(t,x,v) v) + \Lambda(t,x,v) v,
\quad v(0, dx) = v_0(dx).
\end{split}
\end{equation}
In this equation, the first two terms correspond to the Wasserstein component of WFR flows, while the last term plays the role of a Fisher--Rao-type reaction. However, unlike the WFR flow \eqref{eq: WFR}, where both the transport and reaction components are derived from a single energy functional, here the different terms are specified independently and may correspond to unrelated energies.

More importantly, the reaction term \( \Lambda(t,x,v) v \) in \citet{LeCavil2017} depends locally on the solution \( v \): that is, \( \Lambda \) is evaluated using the pointwise value \( v(t,x) \), without reference to the global structure of the distribution. This contrasts with our setting, where, as we will explain in Section \ref{sec: verification}, the reaction term involves a nonlocal dependence on the entire probability measure. Hence, PDE \eqref{eq: Cavil_PDE} does not in general conserve total mass: the solution \( v(t,\cdot) \) is not necessarily a probability density for all \( t \). In contrast, the nonlocal normalization we include ensures mass preservation at all times. Moreover, the regularity assumptions also differ, cf.\ Assumption 1 in  \citet{LeCavil2017} to our Assumption \ref{assumption1}. 

Another key distinction is that we work on a compact state space \( \mathcal{X} \subset \mathbb{R}^d \), whereas the equations in \citet{LeCavil2017} are defined on the whole space \( \mathbb{R}^d \). Extending the results from \cite{LeCavil2017} to compact domains would require working with appropriately defined boundary conditions for the PDE \eqref{eq: Cavil_PDE}.

As we discussed in Remark \ref{remark: Domingo}, another related paper \citep{Domingo-EnrichBruna} studied propagation of chaos for Wasserstein-Fisher-Rao flows without entropy regularization, and for Wassesrstein flows corresponding to the entropy-regularized energy \eqref{eq: energy} (but without the Fisher-Rao part) and hence their results are not applicable to our setting.

Finally, there has been some work on the propagation of chaos for interacting particle systems defined with a killing-replication mechanism with exponential clocks (rather than with evolving weights - cf.\ also the discussion in Section \ref{sec: particle_system}) in \cite{vandenEijnden2019, LuLuNolen2019}. However, we were unable to make the proofs from those works fully rigorous in our setting \eqref{eq: energy}, due to the unboundedness of the flat derivative of the relative entropy.

\section{ENTROPIC MEAN-FIELD PROBLEMS}\label{sec: verification}

In this section, we assume $\mathcal{X}$ to be a compact subset of $\mathbb{R}^d$, and we focus on the energy function
\begin{equation}\label{eq: energy_function}
V^{\sigma}(m) = F(m) + \sigma \operatorname{KL}(m| \pi) \,, \text{ for all } m \in \mathcal{P}(\mathcal{X}).
\end{equation}

Note that the choice of a compact $\mathcal{X}$ removes the technical problem with flat differentiability of the KL-divergence, which on unbounded domains would have to be rigorously justified \cite{Liu-Majka-Szpruch-2023,AubinKorba2022}. Note also that we do not require convexity of $\mathcal{X}$ since the pure Fisher-Rao flow does not change the support of the initial measure, i.e., it evolves only the mass without any transport in the state space.

 The key difficulty in dealing with Fisher-Rao flows \eqref{eq: FR} corresponding to the energy function \eqref{eq: energy_function} lies in the construction of the corresponding interacting particle system. In our setting, one can show that the flat derivative of $V^{\sigma}$ is given by
\begin{equation}\label{eq: RHS_FR_flow}
\frac{\delta V^\sigma}{ \delta \mu} (\mu_t,x ) = \frac{\delta F}{\delta \mu}(\mu_t,x) + \sigma \log \frac{\mu_t(x)}{\pi(x)} - \sigma \operatorname{KL}(\mu_t|\pi).
\end{equation}
Note that the term $- \sigma \operatorname{KL}(\mu_t|\pi)$ is necessary to ensure that the equation \eqref{eq: FR} is conservative, i.e., that all measures $\mu_t$ are indeed probability measures.
However, this makes \eqref{eq: RHS_FR_flow}
well-defined only if the measures $\mu_t$ are absolutely continuous with respect to $\pi$, which  creates problems with defining a particle system approximating \eqref{eq: FR}. Indeed, one cannot evaluate the right hand side of \eqref{eq: FR} at the empirical measure of a corresponding particle system, which makes it necessary to define an auxiliary flow, where the problematic terms are replaced by their counterparts involving convolutions with appropriately defined kernels. This is the rationale behind the introduction of the so-called "kernelized" flows, where instead of the expression in \eqref{eq: RHS_FR_flow}, we work with its kernelized version.

\subsection{Setting for entropic mean-field optimization}\label{sec:preliminaries}

Throughout Section \ref{sec: verification}, we impose the following standing assumptions on the energy functional \(F\), the mollifier kernel \(K_\varepsilon\), and the reference measure \(\pi\).

\begin{assumption}\label{hypo: F_full}
We assume the energy functional \( F : \mathcal{P}(\mathcal{X}) \to \mathbb{R} \) satisfies the following properties:
\begin{enumerate}[label=(\roman*),ref=\theassumption(\roman*)]

     \item\label{hypo:F_full:iii} \( F \) is lower semi-continuous with respect to weak convergence in \(\mathcal{P}_2(\mathcal{X})\).

    \item\label{hypo:F_full:iv} \(F\) is bounded from below: there exists \(F_{\min}\in\mathbb{R}\) such that
    \[
    F(m)\ge F_{\min} \quad \text{for all } m\in \mathcal{P}(\mathcal{X}).
    \]
    
    \item\label{hypo:F_full:i} Flat derivative is bounded: there exists a constant \( C_F > 0 \) such that for all \( \mu \in \mathcal{P}(\mathcal{X}) \) and \( x \in \mathcal{X} \),
    \[
    \left| \frac{\delta F}{\delta \mu}(\mu,x) \right| \leq C_F.
    \]
    
    \item\label{hypo:F_full:ii} Flat derivative is Lipschitz: there exists a constant \( L_F > 0 \) such that for all \( \mu, \nu \in \mathcal{P}(\mathcal{X}) \) and \( x, y \in \mathcal{X} \),
    \[
    \left| \frac{\delta F}{\delta \mu}(\mu,x) - \frac{\delta F}{\delta \mu}(\nu,y) \right| \leq L_F \left( \mathcal{W}_2(\mu, \nu) + |x - y| \right).
    \]

\end{enumerate}
\end{assumption}

Note that these assumptions on $F$
 are standard in the mean-field optimization literature and, in particular, they are satisfied in mean-field models of neural networks studied in papers such as \cite{hu2021mean, ChenRenWang2023}, as well as in mean-field models in policy optimization in reinforcement learning \citep{leahy, LascuMajka2025}.

\begin{assumption}\label{hypo: mollifier_kernel}
Let \(\xi \colon \mathbb{R}^d \to \mathbb{R}_+\) be a smooth, Lipschitz (with constant \(L_\xi\)), radial probability density function with full support \(\mathbb{R}^d\) and finite second moment.  
For \(\varepsilon > 0\), define the rescaled function on the compact space $\mathcal{X}$ by
\[
    \xi_\varepsilon(x) := \frac{1}{C_{\varepsilon,d}}\,\varepsilon^{-d} \, \xi\!\left(\frac{x}{\varepsilon}\right),
\]
where \(C_{\varepsilon,d}\) is the normalization constant ensuring that
\[
    \int_{\mathcal{X}} \xi_\varepsilon(x) \, dx = 1.
\]
In particular, \(\xi_\varepsilon\) is a probability density function on \(\mathcal{X}\). The mollifier kernel \(K_\varepsilon\) on \(\mathcal{X}\) is then defined as
\[
    K_\varepsilon(x) := \big(\xi_\varepsilon * \xi_\varepsilon\big)(x).
\]
\end{assumption}

\begin{assumption}\label{hypo: pi_full}
We assume the reference density \( \pi(x) \) satisfies the following properties:
\begin{enumerate}[label=(\roman*),ref=\theassumption(\roman*)]
    \item\label{hypo:pi_full:i} \( \pi(x) = e^{-U(x)} \) for some continuous potential function \( U : \mathcal{X} \to \mathbb{R} \).
    
    \item\label{hypo:pi_full:iv} \( \pi \) is globally Lipschitz on \( \mathcal{X} \): there exists \( L_\pi>0 \) such that
    \[
    |\pi(x)-\pi(y)| \le L_\pi\,|x-y| \quad \text{for all } x,y\in\mathcal{X}.
    \]

\end{enumerate}
\end{assumption}

Note that under Assumption \ref{hypo: pi_full}, since $\pi \propto e^{-U}$ and $\mathcal{X}$ is compact, there exist constants \( 0 < \pi_{\min} < \pi_{\max} < \infty \) such that
    \begin{equation*}
    \pi_{\min} \leq \pi(x) \leq \pi_{\max}, \quad \text{ for all } x \in \mathcal{X}.
    \end{equation*}

\subsection{Kernelization Strategies}\label{sec: kernelization}
We summarize four kernelization strategies for approximating
Fisher--Rao 
flows corresponding to energy functions \eqref{eq: energy}, based on related strategies that have appeared in the literature in recent years. We will then show in Proposition \ref{proposition: kernelization} that for each of these approaches, the resulting function $a$ satisfies our Assumption \ref{assumption1} and therefore all our results from Section \ref{sec: main} are applicable. Recall our notation for the flow $\partial_t \mu_t = - \mu_t a(\mu_t, \cdot)$ and consider the following choices of $a$ that approximate the flat derivative $\frac{\delta V^{\sigma}}{\delta \mu}$ given in \eqref{eq: RHS_FR_flow}. 

\textbf{Smoothing only the evolving measure \citep{LuLuNolen2019}.}
Here the kernel is applied to $\mu_t$ both inside the logarithm and in the KL divergence. The resulting dynamics read

\begin{equation}\label{eq: kernel_1}
\begin{split}
\partial_t \mu_t
= -\,\mu_t \Biggl(&
\frac{\delta F}{\delta \mu}(\mu_t,\cdot)
+ \sigma\,\log\frac{K_\varepsilon * \mu_t}{\pi} \\
&- \sigma\,\operatorname{KL}(K_\varepsilon * \mu_t \mid \pi)
\Biggr).
\end{split}
\end{equation}

\textbf{Smoothing both the evolving and the target measures \citep{Pampel_2023}.}
In this variant, both $\mu_t$ and $\pi$ are mollified by $K_\varepsilon$, which yields

\begin{equation}\label{eq: kernel_2}
\begin{split}
\partial_t \mu_t
= -\,\mu_t \Biggl(&
\frac{\delta F}{\delta \mu}(\mu_t,\cdot)
+ \sigma\,\log\frac{K_\varepsilon * \mu_t}{K_\varepsilon * \pi} \\
&- \sigma\,\operatorname{KL}(K_\varepsilon * \mu_t \mid K_\varepsilon * \pi)
\Biggr).
\end{split}
\end{equation}

\textbf{Kernelizing the energy via \citet{Lu_2023}.}
Here the kernel is applied already at the level of the energy function, rather than just at the level of the flow; one studies the "true" Fisher--Rao gradient flow corresponding to the modified energy function
\[
V^{\sigma}_{\varepsilon}(m) = F(m) \;+\; \sigma\int \log\!\frac{K_\varepsilon * m}{\pi}\; m(dx).
\]
The induced dynamics takes the form

\begin{equation}\label{eq: kernel_3}
\begin{split}
&\partial_t \mu_t
= -\,\mu_t \sigma\Biggl(\frac{1}{\sigma}
\frac{\delta F}{\delta \mu}(\mu_t,\cdot)
+ \log\frac{K_\varepsilon * \mu_t}{\pi}\\
&+ K_\varepsilon *
   \Bigl(\frac{\mu_t}{K_\varepsilon * \mu_t}\Bigr)
   -\int \log\!\frac{K_\varepsilon * \mu_t}{\pi}\;\mu_t(dx) - 1
\Biggr).
\end{split}
\end{equation}

\textbf{Kernelizing the energy via \citet{Carrillo2019}.}
Another choice of a kernelized energy function replaces the KL term by $\operatorname{KL}(K_\varepsilon * m \mid \pi)$, leading to the energy
\[
V^{\sigma}_{\varepsilon}(m) = F(m) \;+\; \sigma\,\operatorname{KL}(K_\varepsilon * m \mid \pi),
\]
whose Fisher--Rao gradient flow is

\begin{equation}\label{eq: kernel_4}
\begin{split}
\partial_t \mu_t
= -\,\mu_t \Biggl(
&\frac{\delta F}{\delta \mu}(\mu_t,\cdot)
+ \sigma\,K_\varepsilon *
   \log\!\frac{K_\varepsilon * \mu_t}{\pi} \\
&- \sigma\,\operatorname{KL}(K_\varepsilon * \mu_t \mid \pi)
\Biggr).
\end{split}
\end{equation}

 Both \eqref{eq: kernel_3} and \eqref{eq: kernel_4} preserve the structure of Fisher–Rao flows, i.e, they are genuine Fisher--Rao gradient flows of the corresponding kernelized energy functions. A notable drawback of \eqref{eq: kernel_4}, however, is the presence of nested convolutions and integrals—e.g., the term $K_\varepsilon * \log(K_\varepsilon * \mu_t/\pi)$—which prevents a full discretization into finite particle sums even when
$\mu_t$ is an empirical measure, which in practice entails higher computational cost (see the discussion in \cite[Section 1]{Carrillo2019}). 

For all kernelizations defined above, we have the following result.

\begin{proposition}\label{proposition: kernelization}
    Suppose Assumptions \ref{hypo: F_full}, \ref{hypo: mollifier_kernel} and \ref{hypo: pi_full} hold. Consider the gradient flow $\partial_t \mu_t = - \mu_t a(\mu_t, \cdot)$ defined via \eqref{eq: kernel_1}, \eqref{eq: kernel_2}, \eqref{eq: kernel_3} or \eqref{eq: kernel_4}. Then $a$ satisfies Assumption \ref{assumption1}.
\end{proposition}

The proof can be found in Appendix \ref{sec: checking_conditions}.

\subsection{Regularization by $\chi^2$-divergence}

Even though the relative entropy is the most popular choice of regularizer in mean-field optimization problems, other choices are possible, and our framework is indeed applicable to energy functions more general than \eqref{eq: energy}. To illustrate this, we discuss briefly energy functions regularized by the $\chi^2$-divergence (with an appropriate kernelization analogous to kernelization \eqref{eq: kernel_4} for the entropy-regularized energy). Recall that given a fixed reference measure $\pi \in \mathcal{P}(\mathcal{X})$, for any measure $\mu \in \mathcal{P}(\mathcal{X})$ absolutely continuous with respect to $\pi$, the $\chi^2$-divergence is defined by $\chi^2(m|\pi) = \int (\frac{dm}{d\pi} - 1)^2 d \pi$. We consider the functional $V^{\sigma}_{\varepsilon}(m) := F(m)+\sigma \chi^2(K_\varepsilon*m\,|\,\pi)$.
Its Fisher--Rao gradient flow is $\partial_t \mu_t
    = -\mu_t a(\mu_t,\cdot)$, where
\begin{equation}\label{def:achisq}
\begin{split}
    a(m,x) &:= \frac{\delta F}{\delta \mu}(m,x)
    + 2\sigma K_\varepsilon\!\ast\!\left(\frac{K_\varepsilon\!\ast\!m}{\pi}\right)(x) \\
    &- 2\sigma \int K_\varepsilon\!\ast\!\left(\frac{K_\varepsilon\!\ast\!m}{\pi}\right)(z)\,m(z)\,dz .
    \end{split}
\end{equation}
We have the following result.
\begin{proposition}
Suppose Assumptions \ref{hypo: F_full}, \ref{hypo: mollifier_kernel} and \ref{hypo: pi_full} hold. Then  the function $a$ defined by \eqref{def:achisq} satisfies Assumption \ref{assumption1}. 
\end{proposition}
The proof can be found in Appendix \ref{subsec: chisq}.

\subsection{Weak-* Convergence of Minimizers under Kernelization}

In the context of mean-field optimization problems \eqref{eq: energy}, an important question related to kernelizations of the corresponding gradient flows is whether those kernelized gradient flows are associated with energy functions whose minimizers are close to the minimizers of the original energy function in \eqref{eq: energy}. In this subsection, we partially answer this question for the kernelized Fisher--Rao gradient flow associated with the kernelization in \eqref{eq: kernel_3}, namely the case where the flow corresponds to the energy function
\[
    V^{\sigma}_{\varepsilon}(m) := F(m) + \sigma \int_{\mathcal{X}} \log\!\frac{K_\varepsilon * m}{\pi}\; dm(x).
\]
We study the limiting behaviour of minimizers as $\varepsilon \to 0$, and prove that any sequence of minimizers of $V^{\sigma}_{\varepsilon}$ admits a subsequence converging in the weak-* topology to a minimizer of the original problem \eqref{eq: energy}.

\begin{theorem}[Weak-* convergence of minimizers]\label{thm: weak_* cv of minimizers}
  Suppose Assumptions \ref{hypo: F_full}, \ref{hypo: mollifier_kernel} and \ref{hypo: pi_full} hold. 
For each $\varepsilon>0$, let $\mu_\varepsilon \in \mathcal{P}_2(\mathcal{X})$ be a minimizer of \(V^{\sigma}_{\varepsilon}\).  
Then there exists a subsequence \((\varepsilon_k)_{k \ge 1}\) such that $\varepsilon_k \to 0$ as $k \to \infty$, and a measure \(\mu \in \mathcal{P}_2(\mathcal{X})\) such that
\[
    \mu_{\varepsilon_k} \stackrel{*}{\rightharpoonup} \mu \quad \text{as } k \to \infty,
\]
and \(\mu\) is a minimizer of \(V^{\sigma}\) given in \eqref{eq: energy}.
\end{theorem}

We would like to remark that a version of Theorem \ref{thm: weak_* cv of minimizers} can also be obtained, under some additional assumptions, on $\mathcal{X} = \mathbb{R}^d$. For the sake of completeness, we discuss the details in Appendix \ref{sec: weak cv of minimizers}, even though in order to apply the remaining results in this paper to the energy function \eqref{eq: energy_function}, we require compactness of $\mathcal{X}$.

\section{ALGORITHM CONSTRUCTION}\label{sec: particle_system}

The practical usage of the discussed gradient flows depends on developing implementable algorithms for approximating the corresponding interacting particle systems \eqref{eq: FR_interacting_system}. In our algorithm, we draw the position of the particles $\tilde{X}^{i,N}$ at the beginning from the initial distributions $\tilde{\mu}_0^{i,N}$, and afterwards they remain constant. Then we apply a time discretization of the equations for the weights $\tilde{w}^{i,N}$ by an Euler scheme, and hence we obtain a numerical scheme which updates the weights of particles in each step. In the algorithm discussed in this section, the dynamics for the weights follows the equation
\eqref{eq: kernel_1}, for a smooth kernel $K_\varepsilon$.

Combining the theoretical results of the present paper with results guaranteeing convergence of the Fisher-Rao gradient flow to the minimizer $m^{\sigma,*}$ of \eqref{eq: energy} as $t \to \infty$ (see e.g.\ \cite{Liu-Majka-Szpruch-2023}), suggests that for a large number of particles, a large number of iterations, and a small $\varepsilon$, the output of our algorithm should provide a reasonable approximation of $m^{\sigma,*}$. However, a full quantitative analysis of the resulting approximation error remains a challenging open problem for future research. Note that the full analysis would need to take into account the following four errors: the error between the continuous flow and the target (due to running the algorithm for finite time), the error between the "correct" flow and the kernelized flow (due to the use of the kernel $K_\varepsilon$), the error between the continuous particle system and the continuous mean-field flow (due to using a finite number of particles) and the discretisation error for the particle system (due to a positive time-step).

\begin{center}
    \begin{tabular}{l}
    \hline \textbf{Algorithm:} Fisher--Rao gradient descent \\
    \hline
       \vtop{\hbox{\strut \textbf{Input:} particles $X^{i} \sim \mu_0^{i}$ with uniform weights} \hbox{\strut \hspace{0.2 cm} $w_0^{i} = 1$, for $i=1,\ldots,N$}
       \hbox{\strut \hspace{0.2 cm} $\Delta t = $ time interval, $J =$ number of iterations}} \\
       \vtop{\hbox{\strut \textbf{Steps:} update weights}
       \hbox{\strut \hspace{0.2cm} \textbf{for} $j=1:J$ \textbf{do}}
       \hbox{\strut \hspace{0.4cm} $\tilde{\mu}_{j-1}^N = \frac{1}{N} \sum_{l=1}^N w_{j-1}^{l} \delta_{X^{l}}$}
       \hbox{\strut \hspace{0.4cm} \textbf{for} $i=1:N$ \textbf{do}}
       \hbox{\strut \hspace{0.6 cm} $\tilde{V_j}^i := \frac{\delta F}{\delta \mu}(\tilde{\mu}_{j-1}^N,X^i) $}
       \hbox{\strut \hspace{0.5 cm} $+ \sigma \log \left( \frac{1}{N} \sum_{l=1}^N w_{j-1}^l K_\varepsilon(X^i- X^l) \right)$}
       \hbox{\strut \hspace{0.5 cm} $- \sigma \log \pi(X^i)  - \sigma \frac{1}{N} \sum_{k=1}^N w_{j-1}^k \times$}
       \hbox{\strut \hspace{0.3 cm} $\left( \log \left( \frac{1}{N} \sum_{l=1}^N w_{j-1}^l K_\varepsilon(X^k- X^l) \right) - \log \pi(X^k) \right)$}
       \hbox{\strut \hspace{0.6 cm} $\hat{w}_j^i = w_{j-1}^i \exp \left( - \tilde{V}_j^i \; \Delta t \right) $}
       \hbox{\strut \hspace{0.6 cm} $w_j^i = N \hat{w}_{j}^i / \sum_{l=1}^N \hat{w}_{j}^l $} }\\
       \vtop{\hbox{\strut \textbf{Output:} the weighted empirical distribution of} \hbox{\strut \hspace{0.2 cm} the particles approximates $m^{\sigma, *}$}}\\
       \hline
    \end{tabular}
\end{center}

Note that we need the coefficient $N$ in the update for $w_j^i$ since the weights are supposed to add up to $N$ (cf.\ \eqref{eq: FR_interacting_system}).

We remark that in our setting it is also possible to construct an algorithm corresponding to the Wasserstein-Fisher-Rao gradient flow \eqref{eq: WFRsystem}. This algorithm would have an additional step within the outer loop, corresponding to the diffusion movement of the particles due to the Wasserstein part of the flow (i.e., this step would correspond to the discretisation of the SDE in \eqref{eq: WFRsystem}). From a practical point of view, such an algorithm is expected to perform better than the algorithm corresponding to the "pure" Fisher-Rao flow, due to the additional exploration of the state space provided by the diffusion movement (the main drawback of the pure Fisher-Rao flow is that it does not expand the support of the initial distribution of the particles - see also Remark \ref{remark: pure FR flow}). However, since our theoretical results cover only the case of the pure Fisher-Rao flow (cf.\ Remark \ref{remark: Domingo}), we formulate the algorithm without the diffusion part.

Finally, note that due to the Feynman-Kac formula, there is a possible interpretation of the Fisher-Rao flow \eqref{eq: FR with a} as a Kolmogorov equation for a stochastic process with killing (see \cite[Section 6.7.2]{Applebaum} or \cite[Theorem 5.7.6 and Exercise 5.7.10]{KaratzasShreve}). This leads to an alternative idea for constructing a particle system approximating \eqref{eq: FR with a}, which instead of rebalancing weights, uses killing and replication mechanism with appropriately defined exponential clocks. Such algorithms were used in \cite{LuLuNolen2019, Pampel_2023, vandenEijnden2019}, however, we were unable to provide a rigorous proof for propagation of chaos for such particle systems in our framework \eqref{eq: energy}, which is why in this paper we are working with system \eqref{eq: FR_interacting_system}.

\begin{remark}\label{remark: pure FR flow}
As the final remark on the practical applicability of the Fisher-Rao algorithm presented in this section, we would like to stress that the main issue with the pure Fisher-Rao flow is that it is very sensitive to initialization, which is reflected in the theoretical results in papers such as \cite{LuLuNolen2019, Lu_2023, Liu-Majka-Szpruch-2023}, which state that, in order for the FR flow to converge (exponentially) to the target $\mu^*$, the initialization $\mu_0$ has to satisfy a "warm start" condition, i.e., there has to exist a constant $C>0$ such that for any $x \in \mathcal{X}$,
$$\frac{d\mu_0}{d\mu^*}(x) \geq C.$$
\end{remark}
In other words, $\mu_0$ and $\mu^*$ 
 have to be sufficiently similar to each other and in particular they need to have matching supports. Since in practice it may be difficult to choose initialization in a way that guarantees this "warm start" condition, especially in very high-dimensional settings, this may limit the applicability of the pure FR flow.

However, a natural idea for constructing a practically feasible algorithm that utilizes the pure FR flow, would be to initially run a different flow to provide an initial exploration of the state space and then to switch to the pure FR flow. For instance, one could first run a particle system approximating the pure Wasserstein flow 
 (from an arbitrary initialization), for a certain amount of time $t_0 >0$
 that guarantees that there exists $C>0$ such that for any $x \in \mathcal{X}$,
$$\frac{d\mu_{t_0}}{d\mu^*}(x) \geq C$$
and then "switch off" the Wasserstein flow and run a particle approximation of the pure FR flow. According to the theory from the papers cited above, a flow like this would converge exponentially to the target (and the corresponding algorithm would be cheaper to run than an algorithm that uses both Wasserstein and FR flows all the time). A fully rigorous analysis of the corresponding algorithm would require propagation of chaos results for both pure Wasserstein flows (which has been covered extensively in the literature, cf.\ the discussion in the introduction) and the pure FR flows (which we provide in the present paper). A full study of such algorithms (and in particular of the question of how to optimally choose $t_0$) will be the topic of our future work.

\bigskip
\newpage

\subsection*{Akcnowledgement}
PL acknowledges funding from the Slovenian Research and Innovation Agency (ARIS) under programme No. P1-0448 and Croatian Science Foundation grant no. 2277.

\bibliographystyle{apalike}
\bibliography{References}

\def\cprime{$'$}
\begin{thebibliography}{}

\bibitem[Ambrosio et~al., 2008]{AmbrosioGigliSavare2008}
Ambrosio, L., Gigli, N., and Savaré, G. (2008).
\newblock {\em Gradient Flows in Metric Spaces and in the Space of Probability
  Measures}.
\newblock Lectures in Mathematics ETH Zürich. Birkhäuser Verlag, Basel, 2
  edition.

\bibitem[Applebaum, 2009]{Applebaum}
Applebaum, D. (2009).
\newblock {\em L\'evy processes and stochastic calculus}, volume 116 of {\em
  Cambridge Studies in Advanced Mathematics}.
\newblock Cambridge University Press, Cambridge, second edition.

\bibitem[Aubin-Frankowski et~al., 2022]{AubinKorba2022}
Aubin-Frankowski, P.-C., Korba, A., and L\'{e}ger, F. (2022).
\newblock Mirror descent with relative smoothness in measure spaces, with
  application to {S}inkhorn and {EM}.
\newblock In {\em Advances in Neural Information Processing Systems},
  volume~35, pages 17263--17275. Curran Associates, Inc.

\bibitem[{Carrillo} et~al., 2024]{carrillo2024}
{Carrillo}, J.~A., {Chen}, Y., {Zhengyu Huang}, D., {Huang}, J., and {Wei}, D.
  (2024).
\newblock {Fisher-Rao Gradient Flow: Geodesic Convexity and Functional
  Inequalities}.
\newblock {\em arXiv e-prints}, page arXiv:2407.15693.

\bibitem[Carrillo et~al., 2019]{Carrillo2019}
Carrillo, J.~A., Craig, K., and Patacchini, F.~S. (2019).
\newblock A blob method for diffusion.
\newblock {\em Calculus of Variations and Partial Differential Equations},
  58(2):53.

\bibitem[Carrillo et~al., 2003]{carrilloMcCannVillani2003}
Carrillo, J.~A., McCann, R.~J., and Villani, C. (2003).
\newblock Kinetic equilibration rates for granular media and related equations:
  entropy dissipation and mass transportation estimates.
\newblock {\em Rev. Mat. Iberoam.}, 19(3):971--1018.

\bibitem[Carrillo et~al., 2016]{Carrillo2016}
Carrillo, J.~A., Patacchini, F.~S., Sternberg, P., and Wolansky, G. (2016).
\newblock Convergence of a particle method for diffusive gradient flows in one
  dimension.
\newblock {\em SIAM Journal on Mathematical Analysis}, 48(6):3708--3741.

\bibitem[Cavil et~al., 2017]{LeCavil2017}
Cavil, A.~L., Oudjane, N., and Russo, F. (2017).
\newblock Particle system algorithm and chaos propagation related to
  non-conservative mckean type stochastic differential equations.
\newblock {\em Stochastics and Partial Differential Equations: Analysis and
  Computations}, 5(1):1--37.

\bibitem[Chen et~al., 2024]{ChenLinRenWang2024}
Chen, F., Lin, Y., Ren, Z., and Wang, S. (2024).
\newblock Uniform-in-time propagation of chaos for kinetic mean field langevin
  dynamics.
\newblock {\em Electronic Journal of Probability}, 29(none).

\bibitem[Chen et~al., 2023]{ChenRenWang2023}
Chen, F., Ren, Z., and Wang, S. (2023).
\newblock Entropic fictitious play for mean field optimization problem.
\newblock {\em J. Mach. Learn. Res.}, 24:Paper No. [211], 36.

\bibitem[Chen et~al., 2025]{ChenRenWang}
Chen, F., Ren, Z., and Wang, S. (2025).
\newblock Uniform-in-time propagation of chaos for mean field {L}angevin
  dynamics.
\newblock {\em Ann. Inst. Henri Poincar\'e{} Probab. Stat.}, 61(4):2357--2404.

\bibitem[Chizat, 2022]{Chizat2022}
Chizat, L. (2022).
\newblock Mean-field langevin dynamics: Exponential convergence and annealing.
\newblock {\em Transactions on Machine Learning Research}.

\bibitem[Delarue and Tse, 2025]{delaruetse2021}
Delarue, F. and Tse, A. (2025).
\newblock Uniform in time weak propagation of chaos on the torus.
\newblock {\em Ann. Inst. Henri Poincar\'e{} Probab. Stat.}, 61(2):1021--1074.

\bibitem[Domingo-Enrich et~al., 2020]{Domingo-EnrichBruna}
Domingo-Enrich, C., Jelassi, S., Mensch, A., Rotskoff, G., and Bruna, J.
  (2020).
\newblock A mean-field analysis of two-player zero-sum games.
\newblock In Larochelle, H., Ranzato, M., Hadsell, R., Balcan, M., and Lin, H.,
  editors, {\em Advances in Neural Information Processing Systems}, volume~33,
  pages 20215--20226. Curran Associates, Inc.

\bibitem[Durmus et~al., 2020]{durmusandreasgullinzimmer2020}
Durmus, A., Eberle, A., Guillin, A., and Zimmer, R. (2020).
\newblock An elementary approach to uniform in time propagation of chaos.
\newblock {\em Proc. Am. Math. Soc.}, 148(12):5387--5398.

\bibitem[Feinberg et~al., 2014]{Feinberg2014}
Feinberg, E.~A., Kasyanov, P.~O., and Zadoianchuk, N.~V. (2014).
\newblock Fatou's lemma for weakly converging probabilities.
\newblock {\em Theory of Probability \& Its Applications}, 58(4):683--689.

\bibitem[Gallouët and Monsaingeon, 2017]{Gallouet2017}
Gallouët, T.~O. and Monsaingeon, L. (2017).
\newblock A jko splitting scheme for kantorovich–fisher–rao gradient flows.
\newblock {\em SIAM Journal on Mathematical Analysis}, 49(2):1100--1130.

\bibitem[{Gu} and {Kim}, 2025]{GuKim2025}
{Gu}, A. and {Kim}, J. (2025).
\newblock {Mirror Mean-Field Langevin Dynamics}.
\newblock {\em arXiv e-prints}, page arXiv:2505.02621.

\bibitem[Guillin and Monmarch{\'e}, 2021]{guillinmonmarche2021}
Guillin, A. and Monmarch{\'e}, P. (2021).
\newblock Uniform long-time and propagation of chaos estimates for mean field
  kinetic particles in non-convex landscapes.
\newblock {\em J. Stat. Phys.}, 185(2):20.
\newblock Id/No 15.

\bibitem[Hu et~al., 2021]{hu2021mean}
Hu, K., Ren, Z., {\v{S}}i{\v{s}}ka, D., and Szpruch, L. (2021).
\newblock Mean-field langevin dynamics and energy landscape of neural networks.
\newblock {\em Annales de l'Institut Henri Poincar{\'e}, Probabilit{\'e}s et
  Statistiques}, 57(4):2043--2065.

\bibitem[Jordan et~al., 1998]{JKO}
Jordan, R., Kinderlehrer, D., and Otto, F. (1998).
\newblock The variational formulation of the {F}okker-{P}lanck equation.
\newblock {\em SIAM J. Math. Anal.}, 29(1):1--17.

\bibitem[Karatzas and Shreve, 1991]{KaratzasShreve}
Karatzas, I. and Shreve, S.~E. (1991).
\newblock {\em Brownian motion and stochastic calculus}, volume 113 of {\em
  Graduate Texts in Mathematics}.
\newblock Springer-Verlag, New York, second edition.

\bibitem[Kerimkulov et~al., 2025]{Kerimkulov2025}
Kerimkulov, B., Leahy, J., {\v{S}}i{\v{s}}ka, D., Szpruch, {\L}., and Zhang, Y.
  (2025).
\newblock A fisher--rao gradient flow for entropy-regularised markov decision
  processes in polish spaces.
\newblock {\em Foundations of Computational Mathematics}.

\bibitem[Lacker, 2018]{Lacker2018}
Lacker, D. (2018).
\newblock Mean field games and interacting particle systems.

\bibitem[Lacker and Le~Flem, 2023]{lackerleflem2023}
Lacker, D. and Le~Flem, L. (2023).
\newblock Sharp uniform-in-time propagation of chaos.
\newblock {\em Probability Theory and Related Fields}, pages 1--38.

\bibitem[{Lascu} and {Majka}, 2025]{LascuMajka2025}
{Lascu}, R.-A. and {Majka}, M.~B. (2025).
\newblock {Non-convex entropic mean-field optimization via Best Response flow}.
\newblock {\em Advances in Neural Information Processing Systems (NeurIPS
  2025)}.

\bibitem[Lascu et~al., 2024]{lascu2024a}
Lascu, R.-A., Majka, M.~B., and Szpruch, L. (2024).
\newblock A fisher-rao gradient flow for entropic mean-field min-max games.
\newblock {\em Transactions on Machine Learning Research}.

\bibitem[{Lascu} et~al., 2024]{lascu2024linearconvergenceproximaldescent}
{Lascu}, R.-A., {Majka}, M.~B., {{\v{S}}i{\v{s}}ka}, D., and {Szpruch}, {\L}.
  (2024).
\newblock {Linear convergence of proximal descent schemes on the Wasserstein
  space}.
\newblock {\em arXiv e-prints}, page arXiv:2411.15067.

\bibitem[Leahy et~al., 2022]{leahy}
Leahy, J.-M., Kerimkulov, B., \v{S}i\v{s}ka, D., and Szpruch, {\L}. (2022).
\newblock Convergence of policy gradient for entropy regularized {MDP}s with
  neural network approximation in the mean-field regime.
\newblock In {\em Proceedings of the 39th International Conference on Machine
  Learning}, volume 162 of {\em Proceedings of Machine Learning Research},
  pages 12222--12252. PMLR.

\bibitem[Liero et~al., 2018]{LieroMielkeSavare2018}
Liero, M., Mielke, A., and Savar\'e, G. (2018).
\newblock Optimal entropy-transport problems and a new
  {H}ellinger-{K}antorovich distance between positive measures.
\newblock {\em Invent. Math.}, 211(3):969--1117.

\bibitem[Liu et~al., 2023]{Liu-Majka-Szpruch-2023}
Liu, L., Majka, M.~B., and Szpruch, L. (2023).
\newblock Polyak–Łojasiewicz inequality on the space of measures and
  convergence of mean-field birth-death processes.
\newblock {\em Appl Math Optim 87, 48 (2023).}

\bibitem[Lu et~al., 2019]{LuLuNolen2019}
Lu, Y., Lu, J., and Nolen, J. (2019).
\newblock Accelerating langevin sampling with birth-death.
\newblock arXiv e-prints.
\newblock arXiv:1905.09863.

\bibitem[Lu et~al., 2023]{Lu_2023}
Lu, Y., Slepčev, D., and Wang, L. (2023).
\newblock Birth–death dynamics for sampling: global convergence,
  approximations and their asymptotics.
\newblock {\em Nonlinearity}, 36(11):5731.

\bibitem[Malrieu, 2001]{malrieu2001}
Malrieu, F. (2001).
\newblock Logarithmic {Sobolev} inequalities for some nonlinear {PDE}'s.
\newblock {\em Stochastic Processes Appl.}, 95(1):109--132.

\bibitem[Mei et~al., 2018]{Mei_2018}
Mei, S., Montanari, A., and Nguyen, P.-M. (2018).
\newblock A mean field view of the landscape of two-layer neural networks.
\newblock {\em Proceedings of the National Academy of Sciences}, 115(33).

\bibitem[Monmarch{\'e}, 2017]{monmarche2017}
Monmarch{\'e}, P. (2017).
\newblock Long-time behaviour and propagation of chaos for mean field kinetic
  particles.
\newblock {\em Stochastic Processes Appl.}, 127(6):1721--1737.

\bibitem[Monmarch\'e et~al., 2024]{MonmarcheRenWang2024}
Monmarch\'e, P., Ren, Z., and Wang, S. (2024).
\newblock Time-uniform log-{S}obolev inequalities and applications to
  propagation of chaos.
\newblock {\em Electron. J. Probab.}, 29:Paper No. 154, 38.

\bibitem[{Nitanda}, 2024]{Nitanda2024}
{Nitanda}, A. (2024).
\newblock {Improved Particle Approximation Error for Mean Field Neural
  Networks}.
\newblock {\em Advances in Neural Information Processing Systems (NeurIPS
  2024)}.

\bibitem[{Nitanda} et~al., 2025]{NitandaLee2025}
{Nitanda}, A., {Lee}, A., {Tan Xing Kai}, D., {Sakaguchi}, M., and {Suzuki}, T.
  (2025).
\newblock {Propagation of Chaos for Mean-Field Langevin Dynamics and its
  Application to Model Ensemble}.
\newblock {\em Forty-second International Conference on Machine Learning
  (ICML2025)}.

\bibitem[Nitanda et~al., 2022]{Nitanda2022}
Nitanda, A., Wu, D., and Suzuki, T. (2022).
\newblock Convex analysis of the mean field langevin dynamics.
\newblock In Camps-Valls, G., Ruiz, F. J.~R., and Valera, I., editors, {\em
  Proceedings of The 25th International Conference on Artificial Intelligence
  and Statistics}, volume 151 of {\em Proceedings of Machine Learning
  Research}, pages 9741--9757. PMLR.

\bibitem[Pampel et~al., 2023]{Pampel_2023}
Pampel, B., Holbach, S., Hartung, L., and Valsson, O. (2023).
\newblock Sampling rare event energy landscapes via birth-death augmented
  dynamics.
\newblock {\em Physical Review E}, 107(2).

\bibitem[Rotskoff et~al., 2019]{vandenEijnden2019}
Rotskoff, G., Jelassi, S., Bruna, J., and Vanden-Eijnden, E. (2019).
\newblock Neuron birth-death dynamics accelerates gradient descent and
  converges asymptotically.
\newblock In Chaudhuri, K. and Salakhutdinov, R., editors, {\em Proceedings of
  the 36th International Conference on Machine Learning}, volume~97 of {\em
  Proceedings of Machine Learning Research}, pages 5508--5517. PMLR.

\bibitem[Salim et~al., 2020]{korbaproximal}
Salim, A., Korba, A., and Luise, G. (2020).
\newblock The {W}asserstein proximal gradient algorithm.
\newblock In {\em Advances in Neural Information Processing Systems},
  volume~33, pages 12356--12366. Curran Associates, Inc.

\bibitem[Schuh, 2024]{schuh2022}
Schuh, K. (2024).
\newblock Global contractivity for {L}angevin dynamics with
  distribution-dependent forces and uniform in time propagation of chaos.
\newblock {\em Ann. Inst. Henri Poincar\'e{} Probab. Stat.}, 60(2):753--789.

\bibitem[{Suzuki} et~al., 2023]{SuzukiNitandaWu2023}
{Suzuki}, T., {Nitanda}, A., and {Wu}, D. (2023).
\newblock {Uniform-in-time propagation of chaos for the mean field gradient
  Langevin dynamics}.
\newblock {\em The Eleventh International Conference on Learning
  Representations (ICLR2023)}.

\bibitem[{Zhu} and {Mielke}, 2024]{zhu2024kernelapproximationfisherraogradient}
{Zhu}, J.-J. and {Mielke}, A. (2024).
\newblock {Kernel Approximation of Fisher-Rao Gradient Flows}.
\newblock {\em arXiv e-prints}, page arXiv:2410.20622.

\end{thebibliography}

\section*{Checklist}

\begin{enumerate}

  \item For all models and algorithms presented, check if you include:
  \begin{enumerate}
    \item A clear description of the mathematical setting, assumptions, algorithm, and/or model. [Yes]
    \item An analysis of the properties and complexity (time, space, sample size) of any algorithm. [Not Applicable]
    \item (Optional) Anonymized source code, with specification of all dependencies, including external libraries. [Not Applicable]
  \end{enumerate}

  \item For any theoretical claim, check if you include:
  \begin{enumerate}
    \item Statements of the full set of assumptions of all theoretical results. [Yes]
    \item Complete proofs of all theoretical results. [Yes]
    \item Clear explanations of any assumptions. [Yes]     
  \end{enumerate}

  \item For all figures and tables that present empirical results, check if you include:
  \begin{enumerate}
    \item The code, data, and instructions needed to reproduce the main experimental results (either in the supplemental material or as a URL). [Not Applicable]
    \item All the training details (e.g., data splits, hyperparameters, how they were chosen). [Not Applicable]
    \item A clear definition of the specific measure or statistics and error bars (e.g., with respect to the random seed after running experiments multiple times). [Not Applicable]
    \item A description of the computing infrastructure used. (e.g., type of GPUs, internal cluster, or cloud provider). [Not Applicable]
  \end{enumerate}

  \item If you are using existing assets (e.g., code, data, models) or curating/releasing new assets, check if you include:
  \begin{enumerate}
    \item Citations of the creator If your work uses existing assets. [Not Applicable]
    \item The license information of the assets, if applicable. [Not Applicable]
    \item New assets either in the supplemental material or as a URL, if applicable. [Not Applicable]
    \item Information about consent from data providers/curators. [Not Applicable]
    \item Discussion of sensible content if applicable, e.g., personally identifiable information or offensive content. [Not Applicable]
  \end{enumerate}

  \item If you used crowdsourcing or conducted research with human subjects, check if you include:
  \begin{enumerate}
    \item The full text of instructions given to participants and screenshots. [Not Applicable]
    \item Descriptions of potential participant risks, with links to Institutional Review Board (IRB) approvals if applicable. [Not Applicable]
    \item The estimated hourly wage paid to participants and the total amount spent on participant compensation. [Not Applicable]
  \end{enumerate}

\end{enumerate}

\clearpage
\onecolumn
\renewcommand{\contentsname}{Table of Contents}
\tableofcontents
\clearpage

\clearpage
\appendix
\thispagestyle{empty}

\onecolumn
\aistatstitle{Appendix: On Propagation of Chaos for the Fisher-Rao Gradient Flow in Entropic Mean Field Optimization}

\section{WELL-POSEDNESS OF FISHER–RAO FLOWS}\label{section: lifted_flow}

\begin{proof}[Proof of Lemma \ref{lemma: projected_lifted}]
We aim to show that if $\nu_t$ solves \eqref{eq: lifted_flow}, then the projected density $\mu_t(x) = \int_0^\infty w \nu_t(x,w) \, dw$ satisfies \eqref{eq: FR_flow} in the weak sense, namely we want to show for all test function $\varphi \in \mathcal{C}^\infty_c(\mathcal{X})$
\begin{equation*}
    \begin{split}
        &\partial_t \int_{\mathcal{X}} \varphi(x) \mu_t(x) dx = - \int_{\mathcal{X}} \varphi(x) \cdot a(\mu_t,x) \cdot \mu_t(x) dx\\
        &\mu_t(x) : = \int_0^\infty w \nu_t(x,w)dw.
    \end{split}.
\end{equation*}

Since we assumed that $\nu_t$ solves \eqref{eq: lifted_flow} weakly, we know that for all $\psi \in \mathcal{C}^\infty_c(\mathcal{X} \times \mathbb{R}_+)$ we have
\begin{equation*}
\partial_t \int_{\mathcal{X} \times \mathbb{R}_+} \psi(x,w) \nu_t(x,w) dx dw 
= - \int_{\mathcal{X} \times \mathbb{R}_+} \frac{\partial \psi}{\partial w}(x,w) \cdot \nu_t(x,w) \cdot w \cdot a(\mu_t,x) \, dx dw.
\end{equation*}
Thanks to Remark \eqref{remark: cc infinity dense in Lp}, we choose $\psi$ in the form of $\psi(x,w) = w \varphi(x)$ for a $\varphi \in \mathcal{C}^\infty_c(\mathcal{X})$. Hence the above equation becomes

\begin{equation*}
\partial_t \int_{\mathcal{X} \times \mathbb{R}_+} 
w \varphi(x) \nu_t(x,w) dx dw 
= - \int_{\mathcal{X} \times \mathbb{R}_+} \varphi(x) \cdot \nu_t(x,w) \cdot w \cdot a(\mu_t,x) \, dx dw.
\end{equation*}
Note that the LHS could be rewrite as:
\begin{equation*}
LHS=\partial_t \int_{\mathcal{X}} \varphi(x) \int_0^\infty w \nu_t(x,w) dw dx 
= \partial_t \int_{\mathcal{X}} \varphi(x) \mu_t(x) dx
\end{equation*}
Moreover, the RHS could be rewrite as:
\begin{equation*}
RHS= - \int_{\mathcal{X}} \varphi(x) \cdot a(\mu_t,x) \left( \int_0^\infty w \nu_t(x,w) dw \right) dx 
= - \int_{\mathcal{X}} \varphi(x) \cdot a(\mu_t,x) \cdot \mu_t(x) dx.
\end{equation*}
Putting it together, we have shown that
\[
\partial_t \int_{\mathcal{X}} \varphi(x) \mu_t(x) dx = - \int_{\mathcal{X}} \varphi(x) \cdot a(\mu_t,x) \cdot \mu_t(x) dx,
\]
This completes the proof.
\end{proof}

\begin{proof}[Proof of Theorem \ref{thm: particle_representation}]
Let us take a test function $f\in \mathcal{C}^\infty_c(\mathcal{X} \times \mathbb{R}_+)$. By chain rule, we have 
\begin{align*}
    f(X,w_t) &= f(X,w_0) + \int_0^t \frac{\partial f}{\partial w}(X,w_s)dw_s\\
    &= f(X,w_0) - \int_0^t \frac{\partial f}{\partial w}(X,w_s)w_sa(\texttt{h}\operatorname{Law}(X, w_s),X)ds
\end{align*}
Now, if we take expectations on both sides, we get
\begin{align*}
    &\int_{\mathcal{X}\times \mathbb{R}_+} f(x,w) \operatorname{Law}(X, w_t)(x,w)dxdw\\
    &= \int_{\mathcal{X}\times \mathbb{R}_+} f(x,w) \operatorname{Law}(X, w_0)(x,w)dxdw\\
    &- \int_0^t \int_{\mathcal{X}\times \mathbb{R}_+} \frac{\partial f}{ \partial w}(x,w) w a (x,\texttt{h}\operatorname{Law}(X, w_s)) \operatorname{Law}(X, w_s)(x,w)dxdw ds
\end{align*}
For the second trem, by using integration by part, we get the following equation
\begin{equation*}
    \begin{split}
    &\int_{\mathcal{X}\times \mathbb{R}_+} f(x,w) \operatorname{Law}(X, w_t)(x,w)dxdw\\
    &= \int_{\mathcal{X}\times \mathbb{R}_+} f(x,w) \operatorname{Law}(X, w_0)(x,w)dxdw\\
    &+\int_0^t \int_{\mathcal{X}\times \mathbb{R}_+} f(x,w)\frac{\partial}{\partial w}\left(\operatorname{Law}(X, w_s)(x,w) w a(\texttt{h}\operatorname{Law}(X, w_s),x)\right) dxdwds
    \end{split}
\end{equation*}
where the last equality shows $\operatorname{Law}(X,w_s)$ is the weak solution in the sense
\begin{equation}
    \begin{cases}
        \partial_t \nu_t = \frac{\partial}{\partial w}(\nu_t(x,w) wa(\texttt{h} \nu_t,x) )\\
        \nu_t \rightharpoonup \nu_0
    \end{cases}
\end{equation}
\end{proof}

\begin{proof}[Proof of Theorem \ref{thm: exist_unique_solution}]
We abbreviate $\mathcal{P}_2(\mathcal{C}([0,T]; \mathcal{X} \times \mathbb{R}_+))$ by $\mathcal{P}_2^T$ and define the mapping
\begin{align*}
    \Psi: \mathcal{P}_2^T &\to \mathcal{P}_2^T,\\
    \nu &\mapsto \operatorname{Law}(X^\nu, w^\nu),
\end{align*}
where, for a fixed $\nu \in \mathcal{P}_2^T$,
\begin{equation}\label{eq: dynamic_with_choosen_nu}
    \begin{split}
        & (X^\nu, w_0^\nu) \sim \nu_0, \\
        & d w_t^\nu = - w_t^\nu \, a\big(\texttt{h} \nu_t, X^\nu\big) \, dt, \quad t \le T.
    \end{split}
\end{equation}
Here $\nu$ is treated as an external input, independent of the process itself, so \eqref{eq: dynamic_with_choosen_nu} is not a McKean–Vlasov equation but an ODE in $(X^\nu, w_t^\nu)$.  
By the classical existence and uniqueness theory for ODEs with Lipschitz drift, the system admits a unique strong solution whenever the coefficients are globally Lipschitz in $(x,w)$.  
In our setting, it suffices to check that
\[
    (x,w) \mapsto w \, a(\texttt{h} \nu_t, x)
\]
is globally Lipschitz. Indeed, for any $(x,w),(x',w') \in \mathcal{X} \times \mathbb{R}_+$,
\begin{align*}
    \lvert w\,a(\texttt{h}\,\nu_t, x) - w'\,a(\texttt{h}\,\nu_t, x') \rvert 
    &\le \lvert w - w'\rvert \,\lvert a(\texttt{h}\,\nu_t, x) \rvert 
        + \lvert w'\rvert \,\lvert a(\texttt{h}\,\nu_t, x) - a(\texttt{h}\,\nu_t, x') \rvert \\
    &\le M_a\,\lvert w - w'\rvert + e^{M_a T} L_a\,\lvert x - x'\rvert,
\end{align*}
where $M_a$ and $L_a$ denote the uniform bound and Lipschitz constant of $a$, respectively.  
Moreover, the boundedness of $a$ implies, via Proposition~\ref{prop: w_bound}, that $w_t^\nu$ is uniformly bounded in $t$, which ensures that the drift in \eqref{eq: dynamic_with_choosen_nu} is indeed globally Lipschitz.

Furthermore,
\[
    \mathbb{E}\left[ \lVert (X^\nu, w^\nu) \rVert_T^2 \right] < \infty,
\]
so for all $\nu \in \mathcal{P}_2^T$, the mapping $\Psi(\nu)$ remains in $\mathcal{P}_2^T$.  

Now let $\nu, \nu' \in \mathcal{P}_2^T$ be such that $\nu_0 = \nu'_0$.  
We associate to them the processes $(X^\nu, w^\nu)$ and $(X^{\nu'}, w^{\nu'})$, which share identical initial conditions almost surely:
\[
    (X^\nu_0, w^\nu_0) = (X^{\nu'}_0, w^{\nu'}_0) \quad \text{a.s.}
\]
Let $\Psi(\nu)$ and $\Psi(\nu')$ denote the probability measures generated by \eqref{eq: dynamic_with_choosen_nu}. We then have:

\begin{align*}
        &\mathbb{E}\left[ \lVert w^\nu -w^{\nu'}\rVert_T^2\right] \leq T \mathbb{E}\left[ \int_0^T \left\lvert w^\nu_s a(\texttt{h}\nu_s,X^\nu) - w^{\nu'}_s a(\texttt{h} \nu'_s,X^{\nu'})\right\rvert^2ds\right]\\
        &\leq 2T \mathbb{E}\left[ \int_0^T \left\lvert w^\nu_s a(\texttt{h} \nu_s,X^\nu) - w^\nu_s a(\texttt{h}\nu'_s, X^{\nu'})\right\rvert^2 ds\right] \\
        &+ 2T \mathbb{E}\left[ \int_0^T \left\lvert w^\nu_s a(\texttt{h}\nu'_s,X^{\nu'}) - w^{\nu'}_s a(\texttt{h} \nu'_s, X^{\nu'})\right\rvert^2 ds\right]\\
        & \leq 2TL_a^2e^{2M_aT}\mathbb{E}\left[\int_0^T \mathcal{W}^2_2\left(\texttt{h}\nu_s,\texttt{h}\nu'_s\right)+\left\lvert X^\nu -X^{\nu'}\right \rvert^2 ds\right] +2TM_a^2 \mathbb{E}\left[\int_0^T \left\lvert w^\nu_s - w^{\nu'}_s \right\rvert^2 ds\right]\\
        &\leq 2TL_a^2e^{2M_aT}\mathbb{E}\left[\int_0^T \mathcal{W}^2_2\left(\texttt{h}\nu_s,\texttt{h}\nu'_s\right)+\left\lvert X^\nu -X^{\nu'}\right \rvert^2 ds\right] +2TM_a^2 \mathbb{E}\left[\int_0^T \left\lVert w^\nu - w^{\nu'} \right\rVert^2_s ds\right]\\
    \end{align*}
    Here the first inequality follows from the Cauchy–Schwarz inequality, 
the second from the elementary bound $(a+b)^2 \leq 2a^2 + 2b^2$, 
and the third from the Lipschitz continuity of $a$ together with the boundedness of $w^\nu$ (Proposition~\ref{prop: w_bound}) 
and the uniform bound $\lvert a \rvert \le M_a$.  
Finally, the last inequality uses the fact that 
$\lvert f(s) \rvert \le \lVert f \rVert_s$ for any $s \le T$.

    Note that we have
    \begin{equation*}
        \mathbb{E} \left[\left\lvert X^\nu_0 - X^{\nu'}_0\right\rvert^2 \right] =0.
     \end{equation*}
     Combining them together, we have
     \begin{align*}
         &\mathbb{E} \left[\left\lVert X^\nu - X^{\nu'}\right\rVert_T^2 +\lVert w^\nu -w^{\nu'}\rVert_T^2\right] = \mathbb{E} \left[\lVert w^\nu -w^{\nu'}\rVert_T^2\right]\\
         &\leq 2TL_a^2e^{2M_aT}\int_0^T \mathcal{W}^2_2\left(\texttt{h}\nu_s,\texttt{h}\nu'_s\right)ds+2TM_a^2 \mathbb{E}\left[\int_0^T \left\lVert w^\nu - w^{\nu'} \right\rVert^2_s ds\right].
     \end{align*}
By applying Grönwall's lemma together with Proposition \ref{prop: w_bound},
and then using 
\(
\mathcal W_2(\nu_s,\nu'_s) \le \mathcal W_{2,s}(\nu,\nu')
\)
for all \(s\in[0,T]\), we obtain
\begin{align*}
    \mathbb{E} \left[ \lVert w^\nu - w^{\nu'}\rVert_T^2 \right]
    & \leq 2TL_a^2e^{4M_aT+2T^2M_a^2} \int_0^T\mathcal{W}^2_2(\nu_s, \nu'_s)ds\\
    & \leq 2TL_a^2e^{4M_aT+2T^2M_a^2} \int_0^T\mathcal{W}^2_{2,s}(\nu, \nu')ds.
\end{align*}

We denote $C:=  2TL_a^2e^{4M_aT+2T^2M_a^2}$. We have
\begin{equation} \label{eq: FR_estimation_nu_nu'}
    \mathbb{E} \left[ \lVert X^\nu - X^{\nu'} \rVert_T^2 +\lVert w^\nu - w^{\nu'}\rVert_T^2 \right] \leq C\int_0^t \mathcal{W}_{2,s}^2(\nu, \nu')  ds.
\end{equation}
 Moreover, since the joint laws of $((X^\nu,w^\nu),(X^{\nu'},w^{\nu'}))$ is one of the coupling between $\Psi(\nu)$ and $\Psi(\nu')$, that is, it belongs to the set $\Pi(\Psi(\nu),\Psi(\nu'))$. We have
\begin{align*}
    \mathcal{W}_{2,T}^2 \left(\Psi(\nu), \Psi(\nu')\right) &= \inf_{\pi \in \Pi((\Psi(\nu), \Psi(\nu')))} \mathbb{E}_{((X^\nu,w^\nu),(X^{\nu'},w^{\nu'})) \sim \pi}\left[ \lVert X^\nu - X^{\nu'} \rVert_T^2 +\lVert w^\nu - w^{\nu'}\rVert_T^2 \right]\\
    & \leq \mathbb{E} \left[ \lVert X^\nu - X^{\nu'} \rVert_T^2 +\lVert w^{\nu} - w^{\nu'}\rVert_T^2 \right]\\
    &\leq C\int_0^T \mathcal{W}_{2,s}^2(\nu, \nu')  ds 
\end{align*}
We conclude that
\begin{equation}
    \mathcal{W}_{2,T}^2 \left(\Psi(\nu), \Psi(\nu')\right) \leq C\int_0^T \mathcal{W}_{2,s}^2(\nu, \nu')  ds.
\end{equation}
Hence for all $\nu \in \mathcal{P}_{2,T}$ we have
\begin{align*}
    \mathcal{W}_{2,T}^2 \left(\Psi^{k+1}(\nu), \Psi^{k}(\nu)\right) &\leq C \int_0^T \mathcal{W}_{2,s}^2 \left(\Psi^{k}(\nu), \Psi^{k-1}(\nu')\right)  ds\\
    & \leq C^k \int_0^T\int_0^{s} \int_0^{s_1}\cdots \int_0^{s_{k-2}} \mathcal{W}_{2,s_{k-1}}^2(\Psi(\nu), \nu') d s_{k-1} \cdots d s_1 ds\\
    & \leq \frac{C^k T^k}{k!}\mathcal{W}_{2,T}^2(\Psi(\nu), \nu).
\end{align*}
As $(\mathcal{P}_{2,T}, \mathcal{W}_{2,T})$ is complete and $\Psi^{k}(\nu)$ is a Cauchy sequence, the sequence $\Psi^{k}(\nu)$ converges to the unique fixed point of $\Psi$.
\end{proof}

\section{PROPAGATION OF CHAOS FOR THE INTERACTING PARTICLE SYSTEM}\label{sec: POC}

\begin{remark}
For the particle system under consideration, the total weight 
\[
\sum_{i=1}^N w^{i,N}_t
\]
is conserved in time. Indeed, using the property of $a$ that
\[
\int a(m,x) \, m(dx) = 0 \quad \text{for any probability measure } m,
\]
we compute
\begin{align*}
    \partial_t \sum_{i=1}^N w^{i,N}_t
    &= \sum_{i=1}^N w_t^{i,N} a(\texttt{h} \nu_t^N, X^{i,N}_t) \\
    &= \int a( \texttt{h} \nu_t^N,x) \sum_{i=1}^N w^{i,N}_t \delta_{X^{i,N}_t}(x) \, dx \\
    &= \left(\sum_{i=1}^N w^{i,N}_t \right) \int a( \texttt{h} \nu_t^N,x) \,\texttt{h} \nu_t^N(dx) = 0,
\end{align*}
which implies $\sum_{i=1}^N w^{i,N}_t = \sum_{i=1}^N w^{i,N}_0$ for all $t \ge 0$. The same property holds for the tilde system $\tilde{\nu}_t^N$.

In this work, we choose $\sum_{i=1}^N w^{i,N}_0 = N$. This convention ensures that, with the empirical measure
\[
\nu_t^N = \frac{1}{N} \sum_{i=1}^N \delta_{(w^{i,N}_t, X^{i,N}_t)},
\]
its projection
\[
\texttt{h} \nu_t^N = \frac{1}{N} \sum_{i=1}^N w^{i,N}_t \, \delta_{X^{i,N}_t}
\]
is a probability measure. The same reasoning applies to $\tilde{\nu}_t^N$ and $\texttt{h} \tilde{\nu}_t^N$.
\end{remark}

We now state an important remark on the convergence result.

\begin{remark}\label{remark: LLN}
Let \(\mathcal{X}\) be a separable metric space and \((X_i)_{i \geq 1}\) be i.i.d.\ \(\mathcal{X}\)-valued random variables with law \(\mu\).  
Let
\[
    \mu^N := \frac{1}{N} \sum_{i=1}^N \delta_{X_i}
\]
be the empirical measure.  
If \(p \geq 1\) and \(\mu \in \mathcal{P}^p(\mathcal{X})\), then
\[
    \mathcal{W}_p(\mu^N, \mu) \to 0 \quad \text{almost surely as } N \to \infty,
\]
and moreover,
\[
    \mathbb{E} \big[ \mathcal{W}_p^p(\mu^N, \mu) \big] \to 0.
\]
A proof of this result can be found in \cite[Corollary~2.14]{Lacker2018}.
\end{remark}

\begin{proof}[Proof of Theorem \ref{thm: POC}]
    Let us rewrite the particle systems \eqref{eq: systerm iid} and \eqref{eq: FR_interacting_system} in integral form. For the independent system we have
    \begin{equation*}
        \begin{cases}
            X_t^i = X_0^i, \\[0.5em]
            w_t^i = w_0^i - \int_0^t w_s^i \, a\big(\texttt{h}\nu_s, X^i\big)\, ds,
        \end{cases}
    \end{equation*}
    and for the interacting system
    \begin{equation*}
        \begin{cases}
            \tilde{X}_t^{i,N} = \tilde{X}_0^{i,N}, \\[0.5em]
            \tilde{w}_t^{i,N} = \tilde{w}_0^{i,N} 
            - \int_0^t \tilde{w}_s^{i,N}\, a\big( \texttt{h}\tilde{\nu}_s^N, \tilde{X}^{i,N}\big)\, ds.
        \end{cases}
    \end{equation*}
    We have
    \begin{align*}
        &\mathbb{E}\left[ \lVert w^i -\tilde{w}^{i,N}\rVert_T^2\right] \leq \mathbb{E}\left[\left|w^i_0 -\tilde{w}^{i,N}_0\right|\right]+ T \mathbb{E}\left[ \int_0^T \left\lvert w^i_s a(\texttt{h}\nu_s^i,X^i) - \tilde{w}^{i,N}_s a(\texttt{h}\tilde{\nu}_s^N, \tilde{X}^{i,N})\right\rvert^2ds\right]\\
        &\leq \mathbb{E}\left[\left|w^i_0 -\tilde{w}^{i,N}_0\right|\right]\\
        &+2T \mathbb{E}\left[ \int_0^T \left\lvert w^i_s a(\texttt{h}\nu_s^i, X^i) - w^i_s a(\texttt{h}\tilde{\nu}_s^N,\tilde{X}^{i,N})\right\rvert^2  +\left\lvert w^i_s a(\texttt{h}\tilde{\nu}_s^N,\tilde{X}^{i,N}) - \tilde{w}^{i,N}_s a(\tilde{\mu}_s^N,\tilde{X}^{i,N})\right\rvert^2 ds\right] \\
        & \leq \mathbb{E}\left[\left|w^i_0 -\tilde{w}^{i,N}_0\right|\right]\\
        &+2TL_a^2e^{2M_aT}\mathbb{E}\left[\int_0^T \mathcal{W}^2_2\left(\texttt{h}\nu_s^i,\texttt{h}\tilde{\nu}_s^N\right)+\left\lvert X^i -\tilde{X}^{i,N}\right \rvert^2 ds\right] +2TM_a^2 \mathbb{E}\left[\int_0^T \left\lvert w_s^i - \tilde{w}^{i,N}_s \right\rvert^2 ds\right]\\
        &\leq \mathbb{E}\left[\left|w^i_0 -\tilde{w}^{i,N}_0\right|\right]+2TL_a^2e^{2M_aT}\mathbb{E}\left[\int_0^T \mathcal{W}^2_2\left(\texttt{h}\nu_s, \texttt{h}\tilde{\nu}_s^N\right)ds\right]\\
        &+(2TM_a^2+ 2TL_a^2e^{2M_aT}) \mathbb{E}\left[\int_0^T \left \lVert X^i -\tilde{X}^{i,N}\right \rVert^2_s+\left\lVert w^i - \tilde{w}^{i,N} \right\rVert^2_s ds\right]\\
    \end{align*}
    where the first inequality follows from integrating the square of the drift term over $[0,T]$ and applying Cauchy–Schwarz inequality. The second inequality uses the standard bound $(a+b)^2 \leq 2a^2 + 2b^2$. The third inequality follows from the boundedness of $a$ by $M_a$, its Lipschitz continuity with constant $L_a$, and the uniform bound $\lvert w^i_s \rvert \le e^{M_a T}$ from Proposition~\ref{prop: w_bound}. Finally, the last inequality is obtained by combining like terms and applying the bound $\lvert w_s\rvert \leq \lVert w\rVert_s$.

    Moreover, since $X^i$ and $\tilde{X}^{i,N}$ are constant in time, their sup–norm difference reduces to the pointwise difference, i.e.,
    \begin{equation*}
         \mathbb{E} \left[\left\lVert X^i - \tilde{X}^{i,N}\right\rVert_T^2 \right] 
            =  \mathbb{E} \left[ \left\lvert X^i - \tilde{X}^{i,N}\right\rvert^2 \right].
    \end{equation*}

     Combining them together, we have 
    \begin{align*}
         &\mathbb{E} \left[\left\lVert X^i - \tilde{X}^{i,N}\right\rVert_T^2 +\lVert w^i -\tilde{w}^{i,N}\rVert_T^2\right]\\
         &\leq 2TL_a^2e^{2M_aT}\mathbb{E}\left[\int_0^T \mathcal{W}_2^2\left(\texttt{h}\nu_s,\texttt{h}\tilde{\nu}_s^N\right)ds\right] + \mathbb{E}\left[ \left|X^i-\tilde{X}^{i,N} \right|^2+ \left| w^i_0 -\tilde{w}^{i,N}_0 \right|^2 \right]\\
         &+\left(2TM_a^2+ 2TL_a^2e^{2M_aT}\right) \mathbb{E}\left[\int_0^T \left \lVert X^i -\tilde{X}^{i,N}\right \rVert^2_s+\left\lVert w^i - \tilde{w}^{i,N} \right\rVert^2_s ds\right]
     \end{align*}
    Let us denote $C_1 = 2TL_a^2e^{2M_aT}$ and $C_2 = 2TM_a^2+ 2TL_a^2e^{2M_aT}$. Remark that both $C_1$ and $C_2$ are independent with $N$ the number of particles.
    By applying Grönwall’s lemma and using the stability estimates from Propositions~\ref{prop: w_bound} and~\ref{prop: W-2t bound W-2}, we obtain:
    \begin{equation}\label{eq: estmation_system}
    \begin{split}
       &\mathbb{E} \left[\left\lVert X^i - \tilde{X}^{i,N}\right\rVert_T^2 +\lVert w^i -\tilde{w}^{i,N}\rVert_T^2\right]\\
       &\leq e^{ C_2 T} \left(C_1\mathbb{E}\left[\int_0^T\mathcal{W}^2_2(\texttt{h}\nu_s, \texttt{h}\tilde{\nu}^N_s)ds\right]+ \mathbb{E}\left[ \left|X^i-\tilde{X}^{i,N} \right|^2+ \left| w^i_0 -\tilde{w}^{i,N}_0 \right|^2 \right]\right)\\
       &\leq e^{ C_2 T} \left(C_1 e^{2M_a T}\mathbb{E}\left[\int_0^T\mathcal{W}^2_2(\nu_s, \tilde{\nu}^N_s)ds\right]+ \mathbb{E}\left[ \left|X^i-\tilde{X}^{i,N} \right|^2+ \left| w^i_0 -\tilde{w}^{i,N}_0 \right|^2 \right]\right)\\
       &\leq e^{C_2 T} \left(C_1 e^{2M_a T}\mathbb{E}\left[\int_0^T\mathcal{W}^2_{2,s}(\nu, \tilde{\nu}^N)ds\right]+ \mathbb{E}\left[ \left|X^i-\tilde{X}^{i,N} \right|^2+ \left| w^i_0 -\tilde{w}^{i,N}_0 \right|^2 \right]\right).
    \end{split}
    \end{equation}
    We also have
    \begin{equation}\label{eq: estimation_Wasserstein}
         \mathbb{E}\left[ \mathcal{W}^2_{2,T}\left(\nu^N, \tilde{\nu}^N\right)\right] 
         \leq \frac{1}{N}\mathbb{E}\left[\sum_{i=1}^N \left\lVert X^i - \tilde{X}^{i,N}\right \rVert_T^2
         +\left\lVert w^i - \tilde{w}^{i,N}\right\rVert_T^2\right],
    \end{equation}
    which follows by choosing the specific coupling \(\frac{1}{N}\sum_{i=1}^N 
        \delta_{\left((w^i,X^i),(\tilde{w}^{i,N},\tilde{X}^{i,N})\right)}\).
    By using the triangle inequality, \eqref{eq: estimation_Wasserstein} and \eqref{eq: estmation_system} in order, we obtain:
    \begin{align*}
        &\mathbb{E}\left[ \mathcal{W}^2_{2,T}\left(\nu, \tilde{\nu}^N\right)\right] \leq 2\mathbb{E}\left[ \mathcal{W}^2_{2,T}\left(\nu, \nu^N\right)\right]+2\mathbb{E}\left[ \mathcal{W}^2_{2,T}\left(\nu^N, \tilde{\nu}^N\right)\right]\\
        &\leq 2\mathbb{E}\left[ \mathcal{W}^2_{2,T}\left(\nu, \nu^N\right)\right] + \frac{2}{N} \mathbb{E}\left[\sum_{i=1}^N \left\lVert X^i - \tilde{X}^{i,N}\right \rVert_T^2+\left\lVert w^i - \tilde{w}^{i,N}\right\rVert_T^2\right]\\
        &\leq 2\mathbb{E}\left[ \mathcal{W}^2_{2,T}\left(\nu, \nu^N\right)\right] \\
        &+2e^{C_2 T} C_1 e^{2M_a T}\mathbb{E}\left[\int_0^T\mathcal{W}^2_{2,s}(\nu, \tilde{\nu}^N)ds\right]+ \frac{2e^{C_2 T}}{N} \sum_{i=1}^N\mathbb{E}\left[ \left|X^i-\tilde{X}^{i,N} \right|^2+ \left| w^i_0 -\tilde{w}^{i,N}_0 \right|^2 \right]
    \end{align*}

    By applying Gronwall's lemma, we have

    \begin{align*}
        &\mathbb{E}\left[ \mathcal{W}^2_{2,T}\left(\nu, \tilde{\nu}^N\right)\right] \\
        &\leq 2 e^{2e^{C_2 T} C_1 e^{2M_a T} T}\left(\mathbb{E}\left[ \mathcal{W}^2_{2,T}\left(\nu, \nu^N\right)\right]+\frac{e^{C_2 T}}{N} \sum_{i=1}^N\mathbb{E}\left[ \left|X^i-\tilde{X}^{i,N} \right|^2+ \left| w^i_0 -\tilde{w}^{i,N}_0 \right|^2 \right] \right)
    \end{align*}

    which yields the result thanks to Remark \ref{remark: LLN} and our assumption on the initial condition \ref{hypo: initial_condition}.
\end{proof}

\begin{proof}[Proof of Corollary \ref{cor: POC}]
By Propositions \ref{prop: w_bound} and \ref{prop: W-2t bound W-2}, we have for all \(t \in [0,T]\),
\[
\mathcal{W}^2_2\left(\mu_t, \tilde{\mu}_t^N\right) 
\leq e^{2M_a} \mathcal{W}^2_2\left(\nu_t, \tilde{\nu}_t^N\right) 
\leq e^{2M_a} \mathcal{W}^2_{2,T}\left(\nu, \tilde{\nu}^N\right).
\]
Since the right-hand side does not depend on \(t\), taking the supremum over \(t \in [0,T]\) yields
\[
\sup_{t \in [0,T]}\mathcal{W}^2_2\left(\mu_t, \tilde{\mu}_t^N\right) 
\leq e^{2M_a} \mathcal{W}^2_{2,T}\left(\nu, \tilde{\nu}^N\right).
\]
Applying Theorem \ref{thm: POC} then gives the desired result.
\end{proof}

\section{WEAK-* CONVERGENCE OF KERNELIZED MINIMIZERS}\label{sec: weak cv of minimizers}

In this section, we investigate the variational properties of the kernelized energy
\[
V^\sigma_\varepsilon(m)
=
F(m)
+
\sigma
\int \log\!\left( \frac{K_\varepsilon * m}{\pi} \right)\, dm,
\]
with particular emphasis on the existence of minimizers and their behaviour
as the regularization parameter $\varepsilon \to 0$.

We will prove Theorem \ref{thm: weak_* cv of minimizers} under Assumptions \ref{hypo: F_full}, \ref{hypo: mollifier_kernel} and \ref{hypo: pi_full} for compact $\mathcal{X} \subset \mathbb{R}^d$, and also formulate an additional result for $\mathcal{X} = \mathbb{R}^d$.

More precisely, for $\mathcal{X} = \mathbb{R}^d$, we need the following growth condition on the potential $U$ (recall that $\pi \propto e^{-U}$).

\begin{assumption}\label{hypo:pi_noncompact}
We assume that there exist constants
$C_0,A_0\in\mathbb R$ and $C_1,A_1>0$ such that
\[
C_0 + C_1 |x|^2 \le U(x) \le A_0 + A_1 |x|^2,
\qquad \forall x\in\mathbb R^d.
\]
\end{assumption}

Then we have the following additional result.

\begin{theorem}[Weak-* convergence of minimizers - Non Compact case]\label{thm: weak_* cv of minimizers non compact}
Suppose $\mathcal{X} = \mathbb{R}^d$, let Assumptions \ref{hypo: F_full}, \ref{hypo: mollifier_kernel}, \ref{hypo: pi_full} and \ref{hypo:pi_noncompact} hold, and suppose the kernel $\xi$ is Gaussian. 
 Then there exists a minimizer of $V_\varepsilon^\sigma$ in $\mathcal{P}_2(\mathbb{R}^d)$ for all $\varepsilon>0$. Moreover, for all $\varepsilon>0$, if $\mu_\varepsilon \in \mathcal{P}_2(\mathbb{R}^d)$ is a minimizer of $V_\varepsilon^\sigma$ then there exists a subsequence \((\varepsilon_k)_{k \ge 1}\) such that $\varepsilon_k \to 0$ as $k \to \infty$ such that
\[
    \mu_{\varepsilon_k} \stackrel{*}{\rightharpoonup} \mu \quad \text{as } k \to \infty,
\]
where \(\mu\) is a minimizer of \(V^{\sigma}\).
\end{theorem}

We will split the proof into the compact case (Theorem \ref{thm: weak_* cv of minimizers}) and the non-compact case (Theorem \ref{thm: weak_* cv of minimizers non compact}). We will follow \cite{Carrillo2019}.

\subsection{Existence of minimizers - Compact case}

\begin{proposition}\label{eq: V_varepsilon_lsc}
Suppose Assumptions \ref{hypo: F_full}, \ref{hypo: mollifier_kernel} and \ref{hypo: pi_full} hold. Then for all \(\varepsilon > 0\), the energy functional
\[
V^\sigma_\varepsilon(m) := F(m) + \sigma \int_{\mathcal{X}} \log \left( \frac{K_\varepsilon * m}{\pi} \right) \, dm
\]
is lower semi-continuous with respect to weak convergence in \(\mathcal{P}(\mathcal{X})\).
\end{proposition}

\begin{proof}
Let \((m_n)_{n \in \mathbb{N}} \subset \mathcal{P}(\mathcal{X})\) be a sequence of probability measures converging weakly to \(m \in \mathcal{P}(\mathcal{X})\). We aim to show
\[
\liminf_{n \to \infty} \left( F(m_n) + \sigma \int \log \left( \frac{K_\varepsilon * m_n}{\pi} \right) \, dm_n \right)
\geq F(m) + \sigma \int \log \left( \frac{K_\varepsilon * m}{\pi} \right) \, dm.
\]

Since \(F\) is lower semi-continuous and \(\pi\) is continuous and bounded from below by assumption we know that $m \mapsto F(m) - \sigma\int \log \pi dm$ is lower semi-continuous. Hence it suffices to establish
\[
\liminf_{n \to \infty} \int \log (K_\varepsilon * m_n) \, dm_n \geq \int \log (K_\varepsilon * m) \, dm.
\]

Note that for fixed \(\varepsilon > 0\), the function \(x \mapsto \log (K_\varepsilon * m_n(x))\) is continuous and bounded from below by \(\log K_{\min,\varepsilon}\), where $K_{\min,\varepsilon} := \inf_{x \in \mathcal{X}} K_\varepsilon(x)$. Therefore, we can apply a version of the generalized Fatou's lemma (see, e.g., \citet{Feinberg2014}), which yields:
\begin{equation} \label{eq: modified_Fatou}
\liminf_{n \to \infty} \int \log \left( \frac{K_\varepsilon * m_n(x)}{K_{\min,\varepsilon}} \right) \, dm_n(x) 
\geq \int \liminf_{n \to \infty, x' \to x} \log \left( \frac{K_\varepsilon * m_n(x')}{K_{\min,\varepsilon}} \right) \, dm(x).
\end{equation}

Moreover, for each \(x \in \mathcal{X}\), the convolution \(K_\varepsilon * m_n(x')\) converges pointwise to \(K_\varepsilon * m(x)\) as \(n \to \infty\) and \(x' \to x\), due to weak convergence and the continuity of the kernel. Therefore,
\begin{equation} \label{eq: pointwise_cv}
\liminf_{n \to \infty, x' \to x} \log \left( \frac{K_\varepsilon * m_n(x')}{K_{\min,\varepsilon}} \right) 
= \log \left( \frac{K_\varepsilon * m(x)}{K_{\min,\varepsilon}} \right).
\end{equation}

Combining \eqref{eq: modified_Fatou} and \eqref{eq: pointwise_cv}, we conclude that
\[
\liminf_{n \to \infty} \int \log (K_\varepsilon * m_n) \, dm_n \geq \int \log (K_\varepsilon * m) \, dm,
\]
which proves the lower semicontinuity of \(V^\sigma_\varepsilon\).
\end{proof}

\begin{proof}[Proof of Theorem \ref{thm: weak_* cv of minimizers} (Step 1: Existence of minimizers)]
For a compact domain $\mathcal X$, the argument is straightforward.
The energy $V^\sigma_\varepsilon$ is bounded from below.
Moreover, any sequence of probability measures (in particular, any minimizing sequence)
is automatically tight as a direct consequence of compactness.
Indeed, for any $p \ge 1$,
\begin{equation*}
    \int_{\mathcal X} |x|^p \, dm(x)
    \le
    \sup_{x \in \mathcal X} |x|^p .
\end{equation*}
Together with the lower semicontinuity established in
Proposition~\ref{eq: V_varepsilon_lsc},
this implies the existence of a weakly-$*$ convergent subsequence whose limit
is a minimizer of $V^\sigma_\varepsilon$.
\end{proof}

\subsection{Existence of minimizers - Non Compact case}

\begin{proposition}\label{eq: V_varepsilon_lsc}
Suppose Assumptions \ref{hypo: F_full}, \ref{hypo: mollifier_kernel}, \ref{hypo: pi_full} and \ref{hypo:pi_noncompact} hold and the kernel $\xi$ is Gaussian. Then for all \(\varepsilon > 0\), the energy functional
\[
V^\sigma_\varepsilon(m) := F(m) + \sigma \int_{\mathbb{R}^d} \log \left( \frac{K_\varepsilon * m}{\pi} \right) \, dm
\]
is lower semi-continuous with respect to weak convergence in \(\mathcal{P}_2(\mathbb{R}^d)\).
\end{proposition}

\begin{proof}
Let \((m_n)_{n \in \mathbb{N}} \subset \mathcal{P}_2(\mathbb{R}^d)\) be a sequence of probability measures converging weakly to some \(m \in \mathcal{P}_2(\mathbb{R}^d)\). We aim to show
\[
\liminf_{n \to \infty} \left( F(m_n) + \sigma \int \log \left( \frac{K_\varepsilon * m_n}{\pi} \right) \, dm_n \right)
\geq F(m) + \sigma \int \log \left( \frac{K_\varepsilon * m}{\pi} \right) \, dm.
\]

Since \(F\) is lower semi-continuous and $\log\pi$ is integrable with respect to any $m \in \mathcal{P}_2(\mathbb{R}^d)$ due to the quadratic growth of $U$, it suffices to establish
\[
\liminf_{n \to \infty} \int \log \left( K_\varepsilon * m_n \right) \, dm_n \geq \int \log \left(K_\varepsilon * m\right) \, dm.
\]

Then we want to apply the generalized Fatou's Lemma (which holds for non negative functions). However, unlike in the compact case, when we work on $\mathbb{R}^d$ we do not have a lower bound for the kernel, hence we need a different approach in order to get a non negative function. From \cite{Carrillo2019}[Proposition 3.9] equation (54) we know that if $K$ is Gaussian, then there exists $x_0 \in \mathbb{R}^d$ and $C_0,C_1 \in \mathbb{R}$ such that for $n$ sufficient large we have
\begin{equation*}
    \log (K_\varepsilon * m_n)(x) \geq C_0 \lvert x - x_0  \rvert^2 +C_1.
\end{equation*}
If we denote the LHS by $f_n(x)$ and RHS by $q(x)$, we can use them to apply generalized Fatou and we have
\begin{equation}
\liminf_{n \to \infty} \int (f_n(x) -q(x)) dm_n(x) 
\geq \int \liminf_{n \to \infty, x' \to x} (f_n(x') - q(x')) \, dm(x).
\end{equation}

Since
\begin{equation*}
    \lim_{n\to\infty} \int_{\mathbb{R}^d} \big(-q(x)\big)\, dm_n(x)
=
\int_{\mathbb{R}^d} \big(-q(x)\big)\, dm(x)
=
\int_{\mathbb{R}^d} \liminf_{\substack{n\to\infty x'\to x}} \big(-q(x')\big)\, dm(x),
\end{equation*}
we can cancel the $q$ terms on both sides. This yields the same result as following.

We conclude that
\[
\liminf_{n \to \infty} \int \log (K_\varepsilon * m_n) \, dm_n \geq \int \log (K_\varepsilon * m) \, dm,
\]
which proves the lower semicontinuity of \(V^\sigma_\varepsilon\).
\end{proof}

\begin{lemma}\citep[Lemma 4.1]{Carrillo2016} \label{lemma: entropy_estimation}
    Suppose $\rho$ is a probability density on $\mathbb{R}^d$ with finite second moment $M_2(\rho) := \int_{\mathbb{R}^d} \lvert x \rvert^2\, \rho(x)dx$. Then for all $\delta>0$, we have
    \begin{equation*}
        \int_{\mathbb{R}^d} \log \rho(x) \rho(x)dx \geq -\left(\frac{2\pi}{\delta}\right)^{d/2} - \delta M_2(\rho),
    \end{equation*}
\end{lemma}

\begin{proof}
We split the integration domain:
\begin{align*}
    \int_{\mathbb{R}^d} \log \rho(x) \rho(x)dx 
    &= \int_{\{\log \rho(x) \leq 0\}} \log \rho(x) \rho(x) dx + \int_{\{\log \rho(x) > 0\}} \log \rho(x) \rho(x)dx \\
    &\geq \int_{\{\log \rho(x) \leq 0\}} \log \rho(x) \rho(x)dx \\
    &= -\int_{\{\rho(x) \leq 1\}} \left\lvert \log \rho(x) \right\rvert \rho(x)dx \\
    &= - \int_{\{\rho(x) \leq e^{-\delta |x|^2}\}} \left \lvert\log \rho(x) \right \rvert  \rho(x)dx 
    - \int_{\{e^{-\delta |x|^2} \leq \rho(x) \leq 1\}} \left \lvert \log \rho(x) \right \rvert  \rho(x)dx \\
    &\geq -\delta \int_{\{\rho(x) \leq e^{-\delta |x|^2}\}}\sqrt{\rho(x)} dx
    - \int_{\{e^{-\delta |x|^2} \leq \rho(x) \leq 1\}} |x|^2 \rho(x)dx .
\end{align*}
where we get the last inequality by using the fact that $x|\log(x)| \leq \sqrt{x}$ for $x\in (0,1]$. For the first term, since $\rho(x) \leq 1$, we estimate:
\begin{equation*}
\int_{\{\rho(x) \leq e^{-\delta |x|^2}\}} \sqrt{\rho(x)} dx 
\leq \int_{\mathbb{R}^d} e^{-\delta |x|^2 / 2}\, dx = \left( \frac{2\pi}{\delta} \right)^{d/2}.
\end{equation*}
For the second term, note that
\begin{equation*}
\int_{\{ e^{-\delta |x|^2} \leq \rho(x) \leq 1\}} |x|^2 \rho(x)dx \leq M_2(\rho).
\end{equation*}

Combining the estimates:
\begin{equation*}
\int_{\mathbb{R}^d} \log \rho(x) \rho(x)dx \geq - \delta M_2(\rho) - \left( \frac{2\pi}{\delta} \right)^{d/2}.
\end{equation*}
\end{proof}

\begin{proposition}\label{prop: KL_with_kernel_lower_bound}
    Suppose Assumptions \ref{hypo: F_full}, \ref{hypo: mollifier_kernel} and \ref{hypo: pi_full} hold.
    Then we have for all $\delta > 0$
    \begin{equation}
         \frac{1}{\sigma}V^\sigma_\varepsilon(m) \geq \frac{1}{\sigma}F_{\min} +C_0 +C_1M_2(m)-(2\pi/\delta)^{d/2} - 2\delta \left( M_2(m) + \varepsilon^2 M_2(\xi) \right)
    \end{equation}
\end{proposition}

\begin{proof}
    We begin by rewriting the regularised energy as the sum of three terms:
    \begin{equation*}
        V^\sigma_\varepsilon(m) = F(m) + \sigma \int \log K_\varepsilon* m(x) d m(x) -\sigma \int \log \pi(x) d m(x).
    \end{equation*}
    We estimate each term separately. First, thanks to Assumption \ref{hypo:F_full:iv}
    \begin{equation*}
        F(m) \geq F_{\min}.
    \end{equation*}
    Next, for the third term, note that since $\pi(x) = e^{-U(x)}$, we have
    \begin{equation*}
        - \int \log \pi(x) d m(x) = \int U(x) dm(x) \geq \int C_0 +C_1|x|^2 dm(x) = (C_0 +C_1 M_2(m))
    \end{equation*}
    For the second term, we have for all $\delta > 0$
    \begin{align*}
        \int \log(K_\varepsilon * m )(x)  dm(x) &= \int \log(\xi_\varepsilon *(\xi_\varepsilon* m) )(x)  dm(x)\\
        & = \int \log \left( \int \xi_\varepsilon(y) (\xi_\varepsilon* m) (x-y)dy\right) dm(x)\\
        & \geq \int \left( \int \xi_\varepsilon(y) \log (\xi_\varepsilon* m (x-y))dy\right) dm(x)\\
        &= \int \xi_\varepsilon * \log(\xi_\varepsilon*m)(x) dm(x)\\
        &= \int \log(\xi_\varepsilon*m)(x)  \xi_\varepsilon *m(x)dx \\
        &\geq -(2\pi/\delta)^{d/2} -\delta M_2(\xi_\varepsilon * m).
    \end{align*}
    where the first inequality is from Jensen's inequality and the last inequality is from Lemma \ref{lemma: entropy_estimation}. To control the second moment of the convolution, observe that
    \begin{equation*}
        M_2(\xi_\varepsilon * m) = \int |x|^2 \xi_\varepsilon*m(x)dx 
        = \int |x|^2 \int \xi_\varepsilon(x-y) d m(y) dx\\
        = \iint |y+z|^2 \xi_\varepsilon(z) dz dm(y)
    \end{equation*}
    where in the last step we applied change of variable $z = x-y$. Then by using the inequality \(|x + y|^2 \leq 2|x|^2 + 2|y|^2\) we have

    \begin{align*}
        M_2(\xi_\varepsilon * m)& \leq 2 \iint |z|^2 \xi_\varepsilon(z) dz dm(y) + 2 \iint |y|^2 \xi_\varepsilon(z)dz dm(y)\\
        & = 2 \int |z|^2 \xi_\varepsilon(z) dz + 2 \int |y|^2 dm(y)\\
        & \leq 2\varepsilon^2M_2(\xi)  + 2 M_2(m).
    \end{align*}
    where we used the remark \ref{remark: moment xi varepsilon} in the last step.

    Combining all the above estimates, we conclude that \(V^\sigma_\varepsilon(m)\) is bounded from above by a constant depending on \(\sigma\), \(F_{\min}\), \(\delta\), \(M_2(\xi)\), and \(M_2(m)\).
\end{proof}

\begin{proof}[Proof of Theorem \ref{thm: weak_* cv of minimizers non compact} (Step 1: Existence of minimizers)]

For the non-compact domain $\mathbb{R}^d$, Proposition~\ref{prop: KL_with_kernel_lower_bound} plays a twofold role.
First, it ensures that the energy $V^\sigma_\varepsilon$ is bounded from below
for all $m \in \mathcal P_2(\mathbb{R}^d)$, so that minimizing sequences are well defined.
Second, for any minimizing sequence, choosing $\delta = C_1/4$
makes the coefficient of $M_2(m)$ positive, yielding the estimate
\begin{equation*}
    \frac{1}{\sigma} V^\sigma_\varepsilon(m)
    \ge
    \frac{1}{\sigma} F_{\min}
    + C_0
    + \frac{C_1}{2} M_2(m)
    - (4\pi / C_1)^{d/2}
    - C_1 \varepsilon^2 M_2(\xi).
\end{equation*}
As a consequence, any minimizing sequence $(m_n)_n$
has uniformly bounded second moments.
By Prokhorov's theorem and the lower semicontinuity of $V^\sigma_\varepsilon$,
there exists a weakly-$*$ convergent subsequence whose limit is a minimizer.
\end{proof}

\subsection{Convergence of minimizers (for both cases)}

\begin{lemma}\label{lemma: weak_cv_minimizers}
Let $\mathcal{X} \subset \mathbb{R}^d$ be compact and $\xi$ satisfy Assumption \ref{hypo: mollifier_kernel}, or $\mathcal{X} = \mathbb{R}^d$ and $\xi$ be a Gaussian kernel. 
Let $(\mu_\varepsilon)_{\varepsilon>0}$
be a sequence in $\mathcal{P}(\mathcal{X})$ such that
$\mu_\varepsilon \stackrel{*}{\rightharpoonup} \mu$ as $\varepsilon \to 0$
for some $\mu \in \mathcal{P}(\mathcal{X})$. 
Then
\[
\xi_\varepsilon * \mu_\varepsilon \stackrel{*}{\rightharpoonup} \mu,
\]
where the convergence $\stackrel{*}{\rightharpoonup}$ is understood in the
bounded-Lipschitz sense, i.e.\ tested against all bounded Lipschitz
functions $f:\mathcal{X} \to \mathbb{R}$.

\end{lemma}

\begin{proof}
    Suppose $f$ be a bounded Lipschitz function. We have
    \begin{equation*}
        \left \lvert \int f d(\xi_\varepsilon * \mu_\varepsilon) - \int f d \mu \right \rvert \leq \left \lvert \int f d(\xi_\varepsilon * \mu_\varepsilon) - \int f d \mu_\varepsilon \right \rvert +\left \lvert \int f d \mu_\varepsilon - \int f d \mu \right \rvert.
    \end{equation*}
    The second term converges to $0$, and for the first term we have
    \begin{align*}
        \left \lvert \int f d(\xi_\varepsilon * \mu_\varepsilon) - \int f d \mu_\varepsilon \right \rvert &= \left \lvert \int \int (f(x) - f(y)) \xi_\varepsilon(x-y) dy d \mu_\varepsilon(x) \right \rvert\\
        & \leq \lVert \nabla f\rVert_{L^\infty} \left \lvert \int \int \lvert x-y\rvert \xi_\varepsilon(x-y) dy d\mu_\varepsilon(x)\right \rvert\\
        & = \lVert \nabla f\rVert_{L^\infty}  \int \int \left\lvert z \right\rvert \xi_\varepsilon\left(z\right) dz d \mu_\varepsilon(x)\\
        & = \lVert \nabla f\rVert_{L^\infty} M_1(\xi_\varepsilon).
    \end{align*}
    In the last inequality, the term $M_1(\xi_\varepsilon)$ can be bounded by $\frac{\varepsilon}{C_{\varepsilon,d}}M_1(\xi)$ or $\varepsilon M_1(\xi)$ depending on the the space $\mathcal{X}$ that we work with. Thus the first term also converges to $0$.
\end{proof}

\begin{theorem}\label{thm: GammaCV}
    Suppose $F$ is lower semi-continuous. Our energy $V_\varepsilon^\sigma$ $\Gamma$-converges to $V^\sigma$ in the following sense:\\
    For $(\mu_\varepsilon)_\varepsilon \subset \mathcal{P}(\mathcal{X})$ and $\mu \in \mathcal{P}(\mathcal{X})$,
    \begin{itemize}
        \item if $\mu_\varepsilon \stackrel{*}{\rightharpoonup} \mu$, we have 
        \begin{equation*}
            \liminf_{\varepsilon \to 0} V_\varepsilon^\sigma(\mu_\varepsilon) \geq V^\sigma(\mu)
        \end{equation*}
        \item 
        \begin{equation*}
            \limsup_{\varepsilon \to 0} V_\varepsilon^\sigma(\mu) \leq V^\sigma(\mu)
        \end{equation*}
    \end{itemize}
\end{theorem}

\begin{proof}
    Since for all $\mu$
    \begin{equation*}
        \int \log(K_\varepsilon * \mu )  d\mu \geq \int \xi_\varepsilon * \log(\xi_\varepsilon*\mu) d\mu= \int \log(\xi_\varepsilon*\mu) d \xi_\varepsilon *\mu,
    \end{equation*}
    we have
    \begin{align*}
        \liminf_{\varepsilon \to 0} V_\varepsilon^\sigma(\mu_\varepsilon) \geq \liminf_{\varepsilon \to 0} \left(F(\mu_\varepsilon) +\sigma \operatorname{KL}(\xi_\varepsilon * \mu_\varepsilon |\pi) \right) \geq V^\sigma(\mu)
    \end{align*}
    where the last inequality used l.s.c of $F$ and $\operatorname{KL}$ as well as the weak convergence of $\xi_\varepsilon * \mu_\varepsilon$ that we got from Lemma \ref{lemma: weak_cv_minimizers}. Hence the first item is proven. We have the second item because
    \begin{equation*}
        V^\sigma(\mu) - V_\varepsilon^\sigma(\mu) = \sigma\operatorname{KL}(\mu | K_\varepsilon * \mu) \geq 0.
    \end{equation*}
\end{proof}

\begin{proof}[Proof of Theorem~\ref{thm: weak_* cv of minimizers} and Theorem \ref{thm: weak_* cv of minimizers non compact} (Step 2: Convergence)]
For any $\varepsilon>0$, since $\mu_\varepsilon$ is a minimizer of $V_\varepsilon^\sigma$,
we have for all $\nu$
\[
V_\varepsilon^\sigma(\mu_\varepsilon)\le V_\varepsilon^\sigma(\nu).
\]
Consequently,
\[
\liminf_{\varepsilon\to0}V_\varepsilon^\sigma(\mu_\varepsilon)
\le
\limsup_{\varepsilon\to0}V_\varepsilon^\sigma(\nu)
\le
V^\sigma(\nu),
\]
where the last inequality follows from Theorem~\ref{thm: GammaCV}.
Since there exists $\nu$ such that $V^\sigma(\nu)<\infty$,
the sequence $(V_\varepsilon^\sigma(\mu_\varepsilon))_{\varepsilon>0}$ is uniformly bounded
from above.

Moreover, by the quadratic growth estimate established above, there exist constants
$C_0\in\mathbb R$ and $C_1>0$, independent of $\varepsilon$, such that
\[
V_\varepsilon^\sigma(\mu_\varepsilon)
\ge
F_{\min}
+ C_0
+ \frac{C_1}{2} M_2(\mu_\varepsilon)
- (4\pi/C_1)^{d/2}
- C_1 \varepsilon^2 \,M_2(\xi).
\]
Since the left-hand side is uniformly bounded from above and the last term is
uniformly bounded as $\varepsilon\to0$, it follows that
$(M_2(\mu_\varepsilon))_{\varepsilon>0}$ is uniformly bounded.

If $\mathcal X$ is compact, tightness of $(\mu_\varepsilon)_{\varepsilon>0}$
is immediate.
If we work with $\mathbb{R}^d$, the uniform second-moment bound implies tightness. Therefore, up to a subsequence, $\mu_\varepsilon\rightharpoonup\mu$ weakly
for some $\mu\in\mathcal P_2(\mathcal X)$.

By the $\Gamma$-convergence, we conclude that for all $\nu$
\[
V^\sigma(\mu)
\le
\liminf_{\varepsilon\to0}V_\varepsilon^\sigma(\mu_\varepsilon)
\le
V^\sigma(\nu),
\]
which shows that $\mu$ is a minimizer of $V^\sigma$.
\end{proof}

\section{SUPPLEMENTARY DEFINITIONS AND TECHNICAL RESULTS}

\subsection{Flat Derivative}\label{appendix:firstvariation}

\begin{definition}\label{def fderivative} Fix $q\geq 0$ and let $\mathcal P_q(\mathcal{X})$ be the space of probability measures on $\mathcal{X}$ with finite $q$-th moments. A functional $F:\mathcal P_q(\mathcal{X}) \to \mathbb R$, is said to admit a first order linear derivative (or a flat derivative), if there exists a functional $\frac{\delta F}{\delta m}: \mathcal P_q(\mathcal{X}) \times\mathcal{X}\rightarrow \mathbb R$, such that
\begin{enumerate}
\item For all $a\in \mathcal{X}$, $\mathcal  P_q(\mathcal{X}) \ni m \mapsto \frac{\delta F}{\delta m}(m,a)$ is continuous.
\item For any $\nu \in \mathcal P_q(\mathcal{X})$, there exists $C>0$ such that for all $a\in \mathcal{X}$ we have 
\[
\left|\frac{\delta F}{\delta m}({\nu},a)\right|\leq C(1+|a|^q)\,.
\]
\item For all $m$, $m'\in \mathcal P_q (\mathcal{X})$,
\begin{equation}\label{def FlatDerivative}
F(m')-F(m)=\int_{0}^{1}\int_{\mathcal{X}}\frac{\delta F}{\delta m}(m + \lambda (m'-m),a)\left(m'- m\right)(da)\,d\lambda.
\end{equation}
\end{enumerate}
The functional $\frac{\delta F}{\delta m}$ is then called the linear (functional) derivative of $F$ on $\mathcal P_q(\mathcal{X})$.
\end{definition}

\subsection{Technical Estimates}\label{sec: technical_estimates}

\begin{proposition}[Uniform Bounds and Wasserstein Stability of the Projection]
\label{prop: w_bound}
Suppose Assumption~\ref{eq: boundness_a} holds, and let \(T > 0\) be a finite time horizon. Then the following statements hold for all \(t \in [0, T]\):

\begin{itemize}
    \item[\textnormal{(i)}] \label{item: w_bound_uniform}
    Let \(w_t\) be a solution to the particle dynamics \eqref{eq: mean-field dynamic}. Then \(w_t\) remains uniformly bounded as
    \begin{equation}\label{eq: w_t bound}
        e^{-M_a T} \leq w_t \leq e^{M_a T}.
    \end{equation}

    \item[\textnormal{(ii)}] \label{item: w_bound_wasserstein}
    Let $M>0$ and let \(\nu, \nu' \in \mathcal{P}_2\left(\mathcal{C}([0,T]; \mathcal{X} \times [0,M])\right)\) be two probability measures on the path space. Then their projected measures satisfy the following Wasserstein stability estimate:
    \begin{equation*}
        \mathcal{W}_2^2\left(\texttt{h} \nu_t, \texttt{h} \nu'_t \right) \leq M^2 \, \mathcal{W}_2^2\left( \nu_t, \nu'_t \right), \quad \text{ for all } t \in [0,T].
    \end{equation*}
    In particular, if $\nu$ and $\nu'$ are the laws of solutions to \eqref{eq: mean-field dynamic}, then due to \eqref{eq: w_t bound} we obtain
       \begin{equation}\label{eq: projection_wasserstein_stability}
        \mathcal{W}_2^2\left(\texttt{h} \nu_t, \texttt{h} \nu'_t \right) \leq e^{2M_a T} \, \mathcal{W}_2^2\left( \nu_t, \nu'_t \right), \quad \text{ for all } t \in [0,T].
    \end{equation}
    
\end{itemize}
\end{proposition}

\begin{proof}
    The first item is proven directly by the boundedness assumption. We have 
    \begin{equation*}
        w_t =e^{-\int_0^t a(\texttt{h}\nu_t,X) ds}.
    \end{equation*}
    Hence
    \begin{equation*}
        e^{-M_aT} = e^{\int_0^T -M_adt}\leq  w_t \leq e^{\int_0^T M_a dt} = e^{M_aT}.
    \end{equation*}
    For the second item, let denote $\pi_t$ be any coupling for $\nu_t$ and $\nu'_t$. We denote that $\tilde{\texttt{h}} \pi_t$ is the coupling for $\texttt{h}\nu_t$ and $\texttt{h}\nu'_t$. It is easy to check that $\tilde{\texttt{h}}\pi_t$ is the projection (different to \eqref{eq: projection}, this projection is for product space) of $\pi_t$ in the sense that for any test function $f$ we have
    \begin{align*}
        &\int_{\mathcal{X} \times \mathcal{X}} f(x_t,x_t') \tilde{\texttt{h}}\pi_t(dx_t,dx'_t) \\
        &= \int_{(\mathcal{X} \times [0,M])\times(\mathcal{X} \times [0,M])} w_t w'_t f(x_t,x'_t) \pi_t(dx_t,dw_t,dx'_t,dw'_t) .
    \end{align*}
    Hence we have 
    \begin{align*}
        &\int_{\mathcal{X}\times \mathcal{X}} \lvert x_t -x_t'\rvert^2 \tilde{\texttt{h}}\pi_t(dx_t,dx'_t) \\
        &= \int_{(\mathcal{X} \times [0,M])\times(\mathcal{X} \times [0,M])} w_t w'_t \lvert x_t -x'_t\rvert^2 \pi_t(dx_t,dw_t,dx'_t,dw'_t) \\
        & \leq M^2 \int_{(\mathcal{X} \times [0,M])\times(\mathcal{X} \times [0,M])} \lvert x_t -x'_t\rvert^2 \pi_t(dx_t,dw_t,dx'_t,dw'_t) \\
        & \leq M^2 \int_{(\mathcal{X} \times [0,M])\times(\mathcal{X} \times [0,M])} \left(\lvert x_t -x'_t\rvert^2 + \lvert w_t - w'_t\rvert^2 \right)\pi_t(dx_t,dw_t,dx'_t,dw'_t) 
    \end{align*}
    which yields the result if we take infimum on both sides.
\end{proof}
\begin{proposition}\label{prop: W-2t bound W-2}
Let $\mu$, $\mu' \in \mathcal{P}_2\left(\mathcal{C}([0,T]; \mathcal{X})\right)$. Then, for any $t \in [0,T]$,
    \begin{equation*}
        \mathcal{W}_2(\mu_t,\mu'_t) \leq \mathcal{W}_{2,t}(\mu,\mu').
    \end{equation*}
\end{proposition}
\begin{proof}
  Let $\Pi$ be a coupling of $\mu$ and $\mu'$ and for a fixed $t \in [0,T]$ consider a projection map $\Gamma_t : \mathcal{C}([0,T]; \mathcal{X}) \to \mathcal{X}$ given as $\Gamma_t(\gamma) := \gamma_t$. Let $\pi_t$ be the push-forward of $\Pi$ under $\Gamma_t$. Then
  $$\int_{\mathcal{X} \times \mathcal{X}} |x-y|^2 \pi_t(dx,dy) = \int_{\mathcal{C}([0,T]; \mathcal{X}) \times \mathcal{C}([0,T]; \mathcal{X})} | \gamma_t-\gamma'_t|^2 \Pi(d\gamma,d\gamma').$$
  Note that $\pi_t$ constructed this way is a coupling of $\mu_t = \Gamma_t(\mu)$ and $\mu'_t = \Gamma_t(\mu')$ and hence, using $| \gamma_t-\gamma'_t| \leq \sup_{s\in[0,t]}  | \gamma_s-\gamma'_s|$ and then taking the infimum over $\Pi$ on the right-hand side, we get
  $$\mathcal{W}_2^2(\mu_t,\mu'_t) \leq \int_{\mathcal{X} \times \mathcal{X}} |x-y|^2 \pi_t(dx,dy) \leq \mathcal{W}_{2,t}^2(\mu,\mu').$$
\end{proof}

\subsection{Construction of the Kernel on a compact space}\label{sec: kernel on torus}

Our example is based on the standard Gaussian kernel truncated and renormalized on a compact space. Let $\xi$ denote the standard Gaussian density on $\mathbb{R}^d$:
\[
\xi(z) := \frac{1}{(2\pi)^{d/2}} \exp\left( -\frac{|z|^2}{2} \right).
\]
We then define the mapping 
\[
\xi_\varepsilon : \mathcal{X} \to \mathbb{R}_+, \quad 
\xi_\varepsilon(x) := \frac{1}{C_{\varepsilon,d,L}} \cdot \frac{1}{\varepsilon^d} 
\cdot \xi\!\left( \frac{x}{\varepsilon} \right),
\]
where \(C_{\varepsilon,d,L}\) is the normalization constant ensuring 
\(\int_{\mathcal{X}} \xi_\varepsilon(x)\,dx = 1\).  
We define
\[
C_{\varepsilon,d,L} := \int_{\mathcal{X}} \frac{1}{\varepsilon^d} \xi\left( \frac{x}{\varepsilon} \right) dx
\]
and we have
\[
C_{\varepsilon,d} = \left( 2\Phi\left( \frac{L}{\varepsilon} \right) - 1 \right)^d,
\]
where $\Phi$ denotes the cumulative distribution function of the standard normal distribution. As $\varepsilon \to 0$, we observe that $C_{\varepsilon,d,L} \nearrow 1$. Thus, if we restrict $\varepsilon$ to a compact interval $(0, \varepsilon_{\max}]$, then
\begin{equation}\label{eq: C_varepsilon,d}
C_{\varepsilon,d,L} \in \left[ \left(2\Phi\left( \frac{L}{\varepsilon_{\max}} \right) - 1 \right)^d, 1 \right).
\end{equation}

We now examine the properties of $K_\varepsilon$.

\begin{enumerate}
    \item \textbf{Normalization.} By construction, $\xi_\varepsilon$ is a probability density function on $\mathcal{X}$. Consequently, its convolution is also a probability density function on $\mathcal{X}$.

    \item \textbf{Upper and lower bounds.} Since $\xi$ is smooth and strictly positive, so is $\xi_\varepsilon$. Moreover, since
    \[
        \xi_\varepsilon(x) = \frac{1}{C_{\varepsilon,d}} \cdot \frac{1}{\varepsilon^d} \cdot \frac{1}{(2\pi)^{d/2}} \exp\left( -\frac{|x|^2}{2\varepsilon^2} \right),
    \]
    we obtain the following pointwise bounds for $x \in \mathcal{X}$:
    \begin{align*}
        \sup_{x \in \mathcal{X}} \xi_\varepsilon(x) &\leq \frac{1}{C_{\varepsilon,d}} \cdot \frac{1}{\varepsilon^d} \cdot \frac{1}{(2\pi)^{d/2}}, \\
        \inf_{x \in \mathcal{X}} \xi_\varepsilon(x) &\geq \frac{1}{C_{\varepsilon,d}} \cdot \frac{1}{\varepsilon^d} \cdot \frac{1}{(2\pi)^{d/2}} \cdot \exp\left( -\frac{dL^2}{2\varepsilon^2} \right).
    \end{align*}
    Since $K_\varepsilon = \xi_\varepsilon * \xi_\varepsilon$, the same type of bounds hold for $K_\varepsilon$. In particular, we may introduce the constants
    \[
        K_{\max,\varepsilon} := \sup_{x \in \mathcal{X}} K_\varepsilon(x),\qquad K_{\min,\varepsilon} := \inf_{x \in \mathcal{X}} K_\varepsilon(x),
    \]
    which are strictly positive and finite. Note that $K_{\max,\varepsilon}$ and $K_{\min,\varepsilon}$ depend on $(\varepsilon,d,L)$, but for simplicity of notation we suppress this dependence.

    \item \textbf{Lipschitz continuity.} Since $\xi$ is smooth, we have 
    \begin{align*}
        \lvert \nabla \xi_\varepsilon(x)\rvert 
        &= \frac{1}{C_{\varepsilon,d}} \cdot \frac{1}{\varepsilon^d} \cdot \frac{1}{(2\pi)^{d/2}} 
       \left\lvert \nabla \exp\!\left(-\frac{\lvert x\rvert^2}{2\varepsilon^2}\right)\right\rvert \\
        &= \frac{1}{C_{\varepsilon,d}} \cdot \frac{1}{\varepsilon^d} \cdot \frac{1}{(2\pi)^{d/2}} 
       \frac{\lvert x\rvert}{\varepsilon^2} 
       \exp\!\left(-\frac{\lvert x\rvert^2}{2\varepsilon^2}\right) \\
        &\leq \frac{1}{C_{\varepsilon,d}} \cdot \frac{1}{\varepsilon^{d+1}} \cdot \frac{1}{(2\pi)^{d/2}} \, e^{-1/2},
    \end{align*}
    where the last inequality follows from the fact that $r e^{-r^2/2} \leq e^{-1/2}$ for all $r \geq 0$.

    Therefore, $\xi_\varepsilon$ is globally Lipschitz with constant of order $O(\varepsilon^{-(d+1)})$. 
    Moreover, since 
    \[
        \nabla K_\varepsilon = (\nabla \xi_\varepsilon) * \xi_\varepsilon,
    \]
    we obtain
    \[
        \|\nabla K_\varepsilon\|_\infty \leq \|\nabla \xi_\varepsilon\|_\infty \, \|\xi_\varepsilon\|_{L^1} 
        = \|\nabla \xi_\varepsilon\|_\infty,
    \]
    which shows that $K_\varepsilon$ inherits the same Lipschitz constant. Hence $K_\varepsilon$ is Lipschitz continuous on $\mathcal{X}$ with constant of order $O(\varepsilon^{-(d+1)})$.

\end{enumerate}

\begin{remark}
The kernel \(K_\varepsilon\) defined above is a smooth function on \(\mathcal{X}\), bounded from below and above and globally Lipschitz.  
More precisely, there exist positive constants
\[
    0 < K_{\min,\varepsilon} \leq K_{\max,\varepsilon} < \infty, 
    \quad L_{K_\varepsilon} < \infty,
\]
such that
\begin{equation}\label{eq: l u bound K}
    K_{\min,\varepsilon} \leq K_\varepsilon(x) \leq K_{\max,\varepsilon}
\end{equation}
and
\[
    |K_\varepsilon(x) - K_\varepsilon(y)| \leq L_{K_\varepsilon} \, |x-y|, 
    \quad \forall\, x,y \in \mathcal{X}.
\]
The values of \(K_{\min,\varepsilon}\), \(K_{\max,\varepsilon}\), and \(L_{K_\varepsilon}\) depend on the dimension \(d\), the diameter of $\mathcal{X}$, the parameter \(\varepsilon\), and the choice of the base function \(\xi\). 
\end{remark}

\begin{remark} \label{remark: moment xi varepsilon}
For $p\ge 0$ and an integrable kernel $\xi:\mathbb{R}^d\to[0,\infty)$. Suppose
\(
M_p(\xi):=\int_{\mathbb{R}^d} |u|^p \,\xi(u)\,du
< \infty\).
Let $\xi_\varepsilon$ be as in Assumption~\ref{hypo: mollifier_kernel}.
Then the $p$-th moment of $\xi_\varepsilon$ satisfies
\begin{align*}
M_p(\xi_\varepsilon)
&:=\int_{\mathcal{X}} |x|^p\,\xi_\varepsilon(x)\,dx\\
&=\frac{\varepsilon^p}{C_{\varepsilon,d}}\int_{[-L/\varepsilon,L/\varepsilon]^d} |u|^p\,\xi(u)\,du\\
&\le \frac{\varepsilon^p}{C_{\varepsilon,d}}\,M_p(\xi).
\end{align*}
\end{remark}

\section{VERIFICATION OF THE BOUNDNESS AND LIPSCHITZ CONDITION}\label{sec: checking_conditions}

The main results above—existence and uniqueness (Theorem~\ref{thm: exist_unique_solution}) and propagation of chaos (Theorem~\ref{thm: POC})—require that the drift \(a(m,x)\) be uniformly bounded and Lipschitz continuous in both the spatial variable \(x\) and the measure variable \(m\). We prove that these properties hold for a broad class of kernel-based drifts \(a\). 
The resulting bounds are explicit and depend only on the smoothing kernel, the reference density, and the diameter of \(\mathcal{X}\).

\begin{proposition}\label{prop: a_lipschitz_bounded}
Let \(a(m,x)\) be given by any of the kernelized forms \eqref{eq: kernel_1}–\eqref{eq: kernel_4}. 
Suppose Assumptions~\ref{hypo: F_full}, \ref{hypo: mollifier_kernel}, and \ref{hypo: pi_full} hold. 
Then \(a(m,x)\) satisfies the boundedness and Lipschitz conditions \eqref{eq: boundness_a} and \eqref{eq: Lipschitz_a}.
\end{proposition}

The Lipschitz continuity and boundedness conditions can also be extended to a Fisher--Rao gradient flow for an energy functional with the kernelized chi-square divergence as a regularizer.

\begin{proposition}\label{prop: a_lipschitz_bounded_chi_square}
Consider the energy functional
\[
    F(m) + \sigma \int \left( \frac{K_\varepsilon * m}{\pi} - 1 \right)^2 \pi(x) \, dx.
\]
Its corresponding Fisher--Rao gradient flow takes the form \eqref{eq: FR with a}, where
\begin{equation*}
    a(m,x) = \frac{\delta F}{\delta \mu}(m,x) 
    + 2\sigma \, K_\varepsilon * \left( \frac{K_\varepsilon * m}{\pi} \right)(x) - 2\sigma \int K_\varepsilon * \left( \frac{K_\varepsilon * m}{\pi} \right)(y) \, m(dy).
\end{equation*}
Suppose that Assumptions~\ref{hypo: F_full}, \ref{hypo: mollifier_kernel}, and \ref{hypo: pi_full} hold. 
Then \(a(m,x)\) satisfies both the boundedness condition \eqref{eq: boundness_a} and the Lipschitz continuity condition \eqref{eq: Lipschitz_a}.
\end{proposition}

\begin{proposition}\label{prop:bounded}
For the boundedness condition, we have the following estimates (uniform in $\mu\in\mathcal{P}(\mathcal{X})$ and $x\in\mathcal{X}$):
\begin{enumerate}[label=(\arabic*),ref=\theproposition(\arabic*)]
    \item\label{prop:bounded:1}
    $\left| \frac{\delta F}{\delta \mu}(\mu,x) \right| \le C_F.$

    \item\label{prop:bounded:2}
    $\left| \log \frac{K_\varepsilon * \mu(x)}{\pi(x)} \right|
    \le \max \left( \log\frac{\pi_{\max}}{K_{\min,\varepsilon}}, 
    \log\frac{K_{\max,\varepsilon}}{\pi_{\min}} \right).$

    \item\label{prop:bounded:3}
    $\operatorname{KL}\!\left(K_\varepsilon * \mu  | \pi \right) 
    \le \log \frac{K_{\max,\varepsilon}}{\pi_{\min}}.$

    \item\label{prop:bounded:4}
    $\left\lvert \log \frac{K_\varepsilon*\mu(x)}{K_\varepsilon*\pi(x)} \right\rvert
    \le \max \left( \log\frac{\pi_{\max}}{K_{\min,\varepsilon}}, 
    \log\frac{K_{\max,\varepsilon}}{\pi_{\min}} \right).$

    \item\label{prop:bounded:5}
    $\operatorname{KL}\!\left(K_\varepsilon*\mu | K_\varepsilon*\pi \right)
    \le \log \frac{K_{\max,\varepsilon}}{\pi_{\min}}.$

    \item\label{prop:bounded:6}
     $\displaystyle K_\varepsilon*\left(\frac{\mu}{K_\varepsilon*\mu}\right)(x)
     \le \frac{K_{\max,\varepsilon}}{K_{\min,\varepsilon}}.$

    \item\label{prop:bounded:7}
     $\displaystyle \frac{1}{(K_\varepsilon*\mu)(x)} \le \frac{1}{K_{\min,\varepsilon}}.$
\end{enumerate}
\end{proposition}

\begin{proof}
\ref{prop:bounded:1} follows from Assumption~\ref{hypo: F_full}.
For \ref{prop:bounded:2}, \ref{prop:bounded:3}, \ref{prop:bounded:4}, and \ref{prop:bounded:5}, we use
\[
\frac{K_{\min,\varepsilon}}{\pi_{\max}}
\le \frac{K_\varepsilon*\mu(x)}{\pi(x)}
\le \frac{K_{\max,\varepsilon}}{\pi_{\min}},
\quad \forall x \in \mathcal{X},
\]
which follows from Assumptions~\ref{hypo: mollifier_kernel} and \ref{hypo: pi_full}.
For \ref{prop:bounded:6}, we compute
\begin{align*}
K_\varepsilon*\left(\frac{\mu}{K_\varepsilon*\mu}\right)(x)
&= \int K_\varepsilon(x-y)\left(\frac{\mu}{K_\varepsilon*\mu}\right)(y)\,dy \\
&= \int \frac{K_\varepsilon(x-y)}{K_\varepsilon*\mu(y)}\,\mu(y)\,dy \\
&\le \frac{K_{\max,\varepsilon}}{K_{\min,\varepsilon}}.
\end{align*}
Finally, \ref{prop:bounded:7} is immediate from $(K_\varepsilon*\mu)(x)\ge K_{\min,\varepsilon}$.
\end{proof}

\begin{proposition}\label{prop:lipschitz}
For all \(\mu,\nu\in\mathcal{P}_2(\mathcal{X})\) and \(x,y\in\mathcal{X}\), the following quantitative bounds hold:
\begin{enumerate}[label=(\arabic*),ref=\theproposition(\arabic*)]
    \item\label{prop:lipschitz:1}
    \[
    \left|\frac{\delta F}{\delta \mu}(\mu,x)-\frac{\delta F}{\delta \mu}(\nu,y)\right|
    \le L_F\big(|x-y|+\mathcal{W}_2(\mu,\nu)\big).
    \]

    \item\label{prop:lipschitz:2}
    \[
    \left| \log K_\varepsilon*\mu(x) - \log K_\varepsilon*\nu(x) \right|
    \leq \frac{L_{K_\varepsilon}}{K_{\min,\varepsilon}}\,\mathcal{W}_1(\mu,\nu).
    \]

    \item\label{prop:lipschitz:3}
    \[
    \left| \log K_\varepsilon *\nu(x) - \log K_\varepsilon *\nu(y) \right|
    \leq \frac{L_{K_\varepsilon}}{K_{\min,\varepsilon}}\,|x-y|.
    \]

    \item\label{prop:lipschitz:4}
    \[
    \left| \log K_\varepsilon *\pi(x) - \log K_\varepsilon *\pi(y) \right|
    \leq \frac{L_\pi}{\pi_{\min}}\,|x-y|.
    \]
    
    \item\label{prop:lipschitz:5}
    \[
    \left| \log \frac{K_\varepsilon*\nu(x)}{\pi(x)} 
      - \log \frac{K_\varepsilon*\nu(y)}{\pi(y)} \right|
      \leq \left(\frac{L_{K_\varepsilon}}{K_{\min,\varepsilon}}
      + \frac{L_\pi}{\pi_{\min}}\right)\,|x-y|.
    \]

    \item\label{prop:lipschitz:6}
    \[
    \left| K_\varepsilon* \left(\log \frac{K_\varepsilon*\nu}{\pi}\right)(x)
         - K_\varepsilon* \left(\log \frac{K_\varepsilon*\nu}{\pi}\right)(y)\right|
    \leq \left(\frac{L_{K_\varepsilon}}{K_{\min,\varepsilon}}
      + \frac{L_\pi}{\pi_{\min}}\right)\,|x-y|.
    \]

    \item\label{prop:lipschitz:7}
    \[
    \left| K_\varepsilon* \left(\log \frac{K_\varepsilon*\nu}{K_\varepsilon*\pi}\right)(x)
         - K_\varepsilon* \left(\log \frac{K_\varepsilon*\nu}{K_\varepsilon*\pi}\right)(y)\right|
    \leq \left(\frac{L_{K_\varepsilon}}{K_{\min,\varepsilon}}
      + \frac{L_\pi}{\pi_{\min}}\right)\,|x-y|.
    \]

    \item\label{prop:lipschitz:8}
    \[
    \left|\operatorname{KL}(K_\varepsilon*\mu \mid \pi) 
        - \operatorname{KL}(K_\varepsilon*\nu \mid \pi) \right|
    \leq \left(\frac{2L_{K_\varepsilon}}{K_{\min,\varepsilon}}
      + \frac{L_\pi}{\pi_{\min}}\right)\,\mathcal{W}_1(\mu,\nu).
    \]
    
    \item\label{prop:lipschitz:9}
    \[
    \left| \frac{1}{(K_\varepsilon*\mu)(x)} - \frac{1}{(K_\varepsilon*\mu)(y)} \right|
    \leq \frac{L_{K_\varepsilon}}{K_{\min,\varepsilon}^2}\,|x-y|.
    \]
\end{enumerate}
\end{proposition}

\begin{proof}
For \eqref{prop:lipschitz:1}, we use Assumption~\ref{hypo: F_full}.

For \eqref{prop:lipschitz:2}, 

\begin{equation*}
\begin{aligned}
\left|\log (K_\varepsilon*\mu)(x) - \log (K_\varepsilon*\nu)(x)\right|
&\le \frac{1}{K_{\min,\varepsilon}}\,
   \left|(K_\varepsilon*\mu)(x)-(K_\varepsilon*\nu)(x)\right| \\
&= \frac{1}{K_{\min,\varepsilon}}
   \left| \int K_\varepsilon(x-z)\,(\mu-\nu)(dz) \right| \\
&\le \frac{L_{K_\varepsilon}}{K_{\min,\varepsilon}}\,
   \mathcal{W}_1(\mu,\nu),
\end{aligned}
\end{equation*}

For \eqref{prop:lipschitz:3}, use that $\log$ is $1/K_{\min,\varepsilon}$-Lipschitz on $[K_{\min,\varepsilon},\infty)$ and that $x\mapsto (K_\varepsilon*\nu)(x)$ is $L_{K_\varepsilon}$-Lipschitz. Hence
\begin{align*}
    \left| \log K_\varepsilon *\nu(x) - \log K_\varepsilon *\nu(y) \right|
    &\leq \frac{1}{K_{\min,\varepsilon}}\,
      \left| K_\varepsilon*\nu(x) - K_\varepsilon*\nu(y) \right|\\
      &= \frac{1}{K_{\min,\varepsilon}}\,
      \left| \int \big(K_\varepsilon(z-x)-K_\varepsilon(z-y)\big)\nu(dz)\right|\\
      &\leq \frac{L_{K_\varepsilon}}{K_{\min,\varepsilon}}\,|x-y|.
\end{align*}

For \eqref{prop:lipschitz:4}, again we use that $\log$ is $1/\pi_{\min}$-Lipschitz on $[\pi_{\min},\infty)$ and that $\pi$ is $L_{\pi}$-Lipschitz. Hence
\begin{align*}
    \left| \log K_\varepsilon *\pi(x) - \log K_\varepsilon *\pi(y) \right|
    &\leq \frac{1}{\pi_{\min}}\, \left|  K_\varepsilon *\pi(x) -  K_\varepsilon *\pi(y) \right|\\
    &\leq \frac{1}{\pi_{\min}}\int K_\varepsilon(z)\, |\pi(x-z) - \pi (y-z)|\,dz\\
    &\leq \frac{L_\pi}{\pi_{\min}}\, |x-y|.
\end{align*}

For \eqref{prop:lipschitz:5}, again we use the Lipschitz condition of $\log$ and we have
\begin{equation}
\begin{split}
    \left| 
      \log \frac{K_\varepsilon*\nu(x)}{\pi(x)} 
      - \log \frac{K_\varepsilon*\nu(y)}{\pi(y)} \right|
    &\leq \left| \log K_\varepsilon*\nu(x) - \log K_\varepsilon*\nu(y) \right|
      + \left| \log \pi(x) - \log\pi(y) \right|\\
    &\leq \frac{1}{K_{\min,\varepsilon}}
      \left| K_\varepsilon*\nu(x) - K_\varepsilon*\nu(y) \right|
      + \frac{1}{\pi_{\min}}\, |\pi(x) - \pi(y)|\\
    &= \frac{1}{K_{\min,\varepsilon}}
      \left| \int \big(K_\varepsilon(z-x)-K_\varepsilon(z-y)\big)\nu(dz)\right|
      + \frac{1}{\pi_{\min}}\, |\pi(x) - \pi(y)| \\
    &\leq \left(\frac{L_{K_\varepsilon}}{K_{\min,\varepsilon}}
      + \frac{L_\pi}{\pi_{\min}}\right)\,|x-y|.
\end{split}
\end{equation}

For \eqref{prop:lipschitz:6}, we use \eqref{prop:lipschitz:5} and the fact that
\(\operatorname{Lip}(K_\varepsilon*f) \leq \operatorname{Lip}(f)\).

For \eqref{prop:lipschitz:7}, we use \eqref{prop:lipschitz:3}, \eqref{prop:lipschitz:4}, and
\(\operatorname{Lip}(K_\varepsilon*f) \leq \operatorname{Lip}(f)\).

For \eqref{prop:lipschitz:8}, we have
\begin{align*}
    \left |\operatorname{KL}(K_\varepsilon*\mu \mid \pi) 
        - \operatorname{KL}(K_\varepsilon*\nu \mid \pi) \right|
    &\leq \left|\int 
        \left(\log \frac{K_\varepsilon*\mu(x)}{\pi(x)} 
        - \log \frac{K_\varepsilon*\nu(x)}{\pi(x)}\right) (K_\varepsilon*\mu)(x)\,dx \right| \\
    &\quad + \left|\int  
        \big((K_\varepsilon*\mu)(x) 
        - (K_\varepsilon*\nu)(x)\big) \log \frac{K_\varepsilon*\nu(x)}{\pi(x)}\,dx \right|\\
    &\leq \int \left| 
        \log K_\varepsilon*\mu(x) 
        - \log K_\varepsilon*\nu(x) \right|(K_\varepsilon*\mu)(x)\,dx  \\
    &\quad + \left|\int  
        (\mu(x) - \nu(x))\, K_\varepsilon* \left(\log \frac{K_\varepsilon*\nu(x)}{\pi(x)} \right)\,dx \right|\\
    &\leq \left(\frac{2L_{K_\varepsilon}}{K_{\min,\varepsilon}}
      + \frac{L_\pi}{\pi_{\min}}\right)\,\mathcal{W}_1(\mu,\nu),
\end{align*}
where in the last line we used \eqref{prop:lipschitz:2} and \eqref{prop:lipschitz:4}.

For \eqref{prop:lipschitz:9},
\begin{align*}
    \left| \frac{1}{(K_\varepsilon*\mu)(x)} - \frac{1}{(K_\varepsilon*\mu)(y)} \right|
    &= \left| \frac{(K_\varepsilon*\mu)(x) - (K_\varepsilon*\mu)(y)}{(K_\varepsilon*\mu)(x)(K_\varepsilon*\mu)(y)}\right|\\
    &\leq \frac{1}{K_{\min,\varepsilon}^2}\, \left| (K_\varepsilon*\mu)(x) - (K_\varepsilon*\mu)(y)\right|\\
    &\leq \frac{L_{K_\varepsilon}}{K_{\min,\varepsilon}^2}\,|x-y|.
\end{align*}
\end{proof}

\subsection{\texorpdfstring{For Flow \eqref{eq: kernel_1}}{For flow (1)}}

Define the coefficient associated with the kernelized flow \eqref{eq: kernel_1} by
\[
a(\mu,x)
:= \frac{\delta F}{\delta \mu}(\mu,x)
+ \sigma\Big( (K_\varepsilon * \log\tfrac{K_\varepsilon*\mu}{\pi})(x)
- \operatorname{KL}(K_\varepsilon * \mu | \pi) \Big).
\]
By Proposition~\ref{prop: a_lipschitz_bounded}, \(a\) satisfies both a boundedness estimate
and a Lipschitz estimate. In particular, under Assumptions \ref{hypo: F_full}, \ref{hypo: mollifier_kernel}, and \ref{hypo: pi_full},
there exist constants \(C^1_\varepsilon>0\) and \(L^1_\varepsilon>0\) (given explicitly in
Proposition~\ref{prop: a_lipschitz_bounded}) such that, for all
\(\mu,\nu\in\mathcal{P}(\mathcal{X})\) and \(x,y\in\mathcal{X}\),
\[
|a(\mu,x)| \le C^1_\varepsilon,
\qquad
|a(\mu,x)-a(\nu,y)| \le L^1_\varepsilon\big(\mathcal{W}_2(\mu,\nu)+|x-y|\big).
\]

\begin{proof}
   Thanks to \eqref{prop:bounded:1} and \eqref{prop:bounded:2} we have
\begin{equation*}
    \left| \log \frac{K_\varepsilon * \mu(x)}{\pi(x)} \right|
    \leq \max \left( \log\frac{\pi_{\max}}{K_{\min,\varepsilon}}, 
    \log\frac{K_{\max,\varepsilon}}{\pi_{\min}} \right),
\end{equation*}
and
\begin{equation}
    \operatorname{KL}\left(K_\varepsilon * \mu | \pi \right) 
    \leq \log \frac{K_{\max,\varepsilon}}{\pi_{\min}}.
\end{equation}
Therefore, the desired boundedness condition holds with the constant
\begin{equation*}
    C^1_\varepsilon := C_F + \sigma\max \left( 
    \log\frac{\pi_{\max}}{K_{\min,\varepsilon}}, 
    \log\frac{K_{\max,\varepsilon}}{\pi_{\min}} \right)
    + \sigma\log \frac{K_{\max,\varepsilon}}{\pi_{\min}}.
\end{equation*}

   We now turn to the proof of the Lipschitz condition. 
By the triangle inequality,
\begin{equation*}
    |a(\mu,x)-a(\nu,y)| \leq (i) + \sigma(ii) + \sigma(iii),
\end{equation*}
where
\[
(i) := \left| 
        \frac{\delta F}{\delta \mu}(\mu,x) 
        - \frac{\delta F}{\delta \mu}(\nu,y)\right|,
\quad
(ii) := \left| 
        \log \frac{K_\varepsilon*\mu(x)}{\pi(x)} 
        - \log \frac{K_\varepsilon*\nu(y)}{\pi(y)} \right|,
\]
\[
(iii) := \left| 
        \operatorname{KL}(K_\varepsilon*\mu \mid \pi) 
        - \operatorname{KL}(K_\varepsilon*\nu \mid \pi)\right|.
\]

Thanks to \eqref{prop:lipschitz:1}, we have
\begin{equation}\label{eq: flow1 estination1}
(i) \leq L_F \left( \mathcal{W}_2(\mu,\nu) + \lvert x-y \rvert \right).
\end{equation}

For the term (ii), by triangle inequality we have
\begin{equation*}
    (ii)\leq \left| 
      \log \frac{K_\varepsilon*\mu(x)}{\pi(x)} 
      - \log \frac{K_\varepsilon*\nu(x)}{\pi(x)} \right|
     + \left| 
      \log \frac{K_\varepsilon*\nu(x)}{\pi(x)} 
      - \log \frac{K_\varepsilon*\nu(y)}{\pi(y)} \right|.
\end{equation*}
We bound these two term separately by \eqref{prop:lipschitz:2} and  \eqref{prop:lipschitz:4} and we get:
\begin{equation}\label{eq: flow1 estination2}
    (ii) \leq \frac{L_{K_\varepsilon}}{K_{\min,\varepsilon}}
      \mathcal{W}_1(\mu,\nu) +  \left(\frac{L_{K_\varepsilon}}{K_{\min,\varepsilon}} + \frac{L_\pi}{\pi_{\min}}\right)
      |x-y|.
\end{equation}

Finally, for the terms (iii), thanks to \eqref{prop:lipschitz:5} we have
\begin{equation}\label{eq: flow1 estination3}
    (iii) \leq \left(\frac{2L_{K_\varepsilon}}{K_{\min,\varepsilon}} + \frac{L_\pi}{\pi_{\min}}\right) \mathcal{W}_1(\mu,\nu)
\end{equation}

Combining the bounds from \eqref{eq: flow1 estination1}, \eqref{eq: flow1 estination2}, and \eqref{eq: flow1 estination3}, and using the fact that $\mathcal{W}_1(\mu,\nu) \leq \mathcal{W}_2(\mu,\nu)$, we obtain the desired Lipschitz estimate with the stated constant.

\[
    L^1_\varepsilon := L_F + \sigma\left(\frac{3L_{K_\varepsilon}}{K_{\min,\varepsilon}} + \frac{2L_\pi}{\pi_{\min}}\right).
\]

\end{proof}
\subsection{\texorpdfstring{For Flow \eqref{eq: kernel_2}}{For flow (2)}}

Define the coefficient associated with the kernelized flow \eqref{eq: kernel_2} by
\[
a(\mu,x):=\frac{\delta F}{\delta \mu}(\mu,x)
+ \sigma\Big( \big(K_\varepsilon * \log\tfrac{K_\varepsilon*\mu}{K_\varepsilon*\pi}\big)(x)
- \operatorname{KL}(K_\varepsilon * \mu | K_\varepsilon * \pi) \Big).
\]
By Proposition~\ref{prop: a_lipschitz_bounded}, \(a\) satisfies both a boundedness estimate
and a Lipschitz estimate. In particular, under Assumptions \ref{hypo: F_full}, \ref{hypo: mollifier_kernel}, and \ref{hypo: pi_full},
there exist constants \(C^2_\varepsilon>0\) and \(L^2_\varepsilon>0\) such that, for all
\(\mu,\nu\in\mathcal{P}(\mathcal{X})\) and \(x,y\in\mathcal{X}\),
\[
|a(\mu,x)| \le C^2_\varepsilon,
\qquad
|a(\mu,x)-a(\nu,y)| \le L^2_\varepsilon\big(\mathcal{W}_2(\mu,\nu)+|x-y|\big).
\]

\begin{proof}
Thanks to \eqref{prop:bounded:3} and \eqref{prop:bounded:4} we have
\begin{equation*}
    \left| \log \frac{K_\varepsilon * \mu(x)}{K_\varepsilon*\pi(x)} \right|
    \leq \max \left( \log\frac{\pi_{\max}}{K_{\min,\varepsilon}}, 
    \log\frac{K_{\max,\varepsilon}}{\pi_{\min}} \right),
\end{equation*}
and
\begin{equation}
    \operatorname{KL}\left(K_\varepsilon * \mu | K_\varepsilon*\pi \right) 
    \leq \log \frac{K_{\max,\varepsilon}}{\pi_{\min}}.
\end{equation}
Therefore, the desired boundedness condition holds with the constant
\begin{equation*}
    C^2_\varepsilon := C_F + \sigma\max \left( 
    \log\frac{\pi_{\max}}{K_{\min,\varepsilon}}, 
    \log\frac{K_{\max,\varepsilon}}{\pi_{\min}} \right)
    + \sigma\log \frac{K_{\max,\varepsilon}}{\pi_{\min}}.
\end{equation*}

We now turn to the proof of the Lipschitz condition. 
By the triangle inequality,
\begin{equation*}
    |a(\mu,x)-a(\nu,y)| \leq (i) + \sigma(ii) + \sigma(iii),
\end{equation*}
where
\[
(i) := \left| 
        \frac{\delta F}{\delta \mu}(\mu,x) 
        - \frac{\delta F}{\delta \mu}(\nu,y) \right|,
\]
\[
(ii) := \left| 
        \log \frac{K_\varepsilon*\mu(x)}{K_\varepsilon*\pi(x)} 
        - \log \frac{K_\varepsilon*\nu(y)}{K_\varepsilon*\pi(y)} \right|,
\]
\[
(iii) := \left|
        \operatorname{KL}(K_\varepsilon*\mu \mid K_\varepsilon*\pi) 
        - \operatorname{KL}(K_\varepsilon*\nu \mid K_\varepsilon*\pi) \right|.
\]

    We will show separately that each part is Lipschitz. 
    Thanks to \eqref{prop:lipschitz:1}, we have
\begin{equation}\label{eq: flow2 estination1}
(i) \leq L_F \left( \mathcal{W}_2(\mu,\nu) + \lvert x-y \rvert \right).
\end{equation}
For the term (ii), by triangle inequality we have
    \begin{align*}
         (ii) &\leq \left \lvert \log \frac{K_\varepsilon*\mu(x)}{K_\varepsilon*\pi(x)} -\log \frac{K_\varepsilon*\nu(x)}{K_\varepsilon*\pi(x)}\right \rvert+\left \lvert \log \frac{K_\varepsilon*\nu(x)}{K_\varepsilon*\pi(x)} -\log \frac{K_\varepsilon*\nu(y)}{K_\varepsilon*\pi(y)}\right \rvert
    \end{align*}
    For the first term on RHS, thanks to \eqref{prop:lipschitz:2}we have
    \begin{align*}
        \left \lvert \log \frac{K_\varepsilon*\mu(x)}{K_\varepsilon*\pi(x)} -\log \frac{K_\varepsilon*\nu(x)}{K_\varepsilon*\pi(x)}\right \rvert = \left \lvert \log K_\varepsilon*\mu(x) -\log K_\varepsilon*\nu(x)\right \rvert\leq \frac{L_{K_\varepsilon}}{K_{\min,\varepsilon}}
      \mathcal{W}_1(\mu,\nu).
    \end{align*}
    For the second term on RHS, we have
    \begin{align*}
        &\left \lvert \log \frac{K_\varepsilon*\nu(x)}{K_\varepsilon*\pi(x)} -\log \frac{K_\varepsilon*\nu(y)}{K_\varepsilon*\pi(y)}\right \rvert\\
        & \leq \left|\log K_\varepsilon*\nu(x) - \log K_\varepsilon*\nu(y) \right| + \left| \log K_\varepsilon *\pi(x) - \log K_\varepsilon *\pi(y) \right|\\
        &\leq \left(\frac{L_{K_\varepsilon}}{K_{\min,\varepsilon}} + \frac{L_\pi}{\pi_{\min}}\right)
      |x-y|
    \end{align*}
    where the second inequality follows from \eqref{prop:lipschitz:6} and \eqref{prop:lipschitz:7}. Hence

\begin{equation}\label{eq: flow2 estination2}
    (ii) \leq \frac{L_{K_\varepsilon}}{K_{\min,\varepsilon}}
      \mathcal{W}_1(\mu,\nu) +  \left(\frac{L_{K_\varepsilon}}{K_{\min,\varepsilon}} + \frac{L_\pi}{\pi_{\min}}\right)
      |x-y|.
\end{equation}

    Regarding the term $(iii)$, we have
    \begin{equation}\label{eq: flow2 estination3}
    \begin{split}
        (iii)&:= \left\lvert \operatorname{KL}(K_\varepsilon*\mu|K_\varepsilon*\pi) -\operatorname{KL}(K_\varepsilon*\nu|K_\varepsilon*\pi)\right\rvert\\
        & \leq \left \lvert \int  \left(\log \frac{K_\varepsilon*\mu(x)}{K_\varepsilon*\pi(x)} - \log \frac{K_\varepsilon*\nu(x)}{K_\varepsilon*\pi(x)}\right) K_\varepsilon*\mu(x) dx \right \rvert\\
        &+ \left \lvert\int \left ( K_\varepsilon*\mu(x) -  K_\varepsilon*\nu(x)\right)\log \frac{K_\varepsilon*\nu(x)}{K_\varepsilon*\pi(x)} dx\right \rvert\\
        &  \leq  \int  \left \lvert\log K_\varepsilon*\mu(x) - \log K_\varepsilon*\nu(x)\right \rvert K_\varepsilon*\mu(x) dx \\
        &+ \left \lvert\int \left (\mu(x) -  \nu(x)\right)K_\varepsilon* \left(\log \frac{K_\varepsilon*\nu(x)}{K_\varepsilon*\pi(x)} \right) dx\right \rvert\\
        &\leq \left(\frac{2L_{K_\varepsilon}}{K_{\min,\varepsilon}} + \frac{L_\pi}{\pi_{\min}}\right) \mathcal{W}_1(\mu,\nu),
    \end{split}
    \end{equation}
    where in the last inequality, we used \eqref{prop:lipschitz:2} and  \eqref{prop:lipschitz:7}.
    Combining \eqref{eq: flow2 estination1}, \eqref{eq: flow2 estination2}, and \eqref{eq: flow2 estination3} with the fact that $\mathcal{W}_1(\mu,\nu)\leq \mathcal{W}_2(\mu,\nu)$, we conclude that $L^2_\varepsilon = L^1_\varepsilon$, which completes the proof.

\end{proof}

\subsection{\texorpdfstring{For Flow \eqref{eq: kernel_3}}{For flow (3)}}
Define the coefficient associated with the kernelized flow \eqref{eq: kernel_3} by
\[
a(\mu,x)
:= \frac{\delta F}{\delta \mu}(\mu,x)
+ \sigma\Bigg(
    \log \frac{(K_\varepsilon*\mu)(x)}{\pi(x)}
    + \Big(K_\varepsilon*\frac{\mu}{K_\varepsilon*\mu}\Big)(x)
    - \int_{\mathcal{X}} \log \frac{(K_\varepsilon*\mu)(z)}{\pi(z)}\,\mu(z)\,dz
    - 1
\Bigg).
\]
Under Assumptions \ref{hypo: F_full}, \ref{hypo: mollifier_kernel}, and \ref{hypo: pi_full},
and by Proposition~\ref{prop: a_lipschitz_bounded}, there exist constants
\(C^3_\varepsilon>0\) and \(L^3_\varepsilon>0\) such that, for all
\(\mu,\nu\in\mathcal{P}(\mathcal{X})\) and \(x,y\in\mathcal{X}\),
\[
|a(\mu,x)| \le C^3_\varepsilon,
\qquad
|a(\mu,x)-a(\nu,y)| \le L^3_\varepsilon\big(\mathcal{W}_2(\mu,\nu)+|x-y|\big).
\]

\begin{proof}
  For the boundedness condition, we have already established uniform bounds for 
$\tfrac{\delta F}{\delta \mu}(\mu,x)$ and 
$\log \tfrac{K_\varepsilon*\mu(x)}{\pi(x)}$ from \eqref{prop:bounded:1} and \eqref{prop:bounded:2}. In particular, we recall that
\begin{equation*}
    \left| \frac{\delta F}{\delta \mu} (\mu,x) \right|
    + \left| \log \tfrac{K_\varepsilon*\mu(x)}{\pi(x)} \right|
    + \left| \int \log \tfrac{K_\varepsilon*\mu(x)}{\pi(x)} \,\mu(x)\, dx \right|
    \leq C_F
    + 2\max\!\left( 
        \log\frac{\pi_{\max}}{K_{\min,\varepsilon}},\,
        \log\frac{K_{\max,\varepsilon}}{\pi_{\min}} 
    \right).
\end{equation*}

Moreover, thanks to \eqref{prop:bounded:6}, for positive term
$K_\varepsilon*\!\left(\tfrac{\mu}{K_\varepsilon*\mu}\right)$, we have
\begin{equation*}
    K_\varepsilon*\left(\frac{\mu}{K_\varepsilon*\mu}\right)(x) \leq \frac{K_{\max,\varepsilon}}{K_{\min,\varepsilon}}.
\end{equation*}
Hence we conclude that
\begin{equation*}
    C^3_\varepsilon 
    = C_F
    + 2\sigma \max\!\left(
        \log\frac{\pi_{\max}}{K_{\min,\varepsilon}},\,
        \log\frac{K_{\max,\varepsilon}}{\pi_{\min}}
      \right)
    + \sigma\left( \tfrac{K_{\max,\varepsilon}}{K_{\min,\varepsilon}} + 1 \right).
\end{equation*}
We now turn to the proof of the Lipschitz condition. 
By the triangle inequality,
\begin{equation*}
    |a(\mu,x)-a(\nu,y)| \leq (i) + \sigma(ii) + \sigma(iii) + \sigma(iv),
\end{equation*}
where
\[
(i) := \left|
        \frac{\delta F}{\delta \mu}(\mu,x) 
        - \frac{\delta F}{\delta \mu}(\nu,y) \right|,
\]
\[
(ii) := \left|
        \log \frac{K_\varepsilon*\mu(x)}{\pi(x)} 
        - \log \frac{K_\varepsilon*\nu(y)}{\pi(y)} \right|,
\]
\[
(iii) := \left|
        K_\varepsilon*\!\left(\tfrac{\mu}{K_\varepsilon*\mu}\right)(x) 
        - K_\varepsilon*\!\left(\tfrac{\nu}{K_\varepsilon*\nu}\right)(y) \right|,
\]
\[
(iv) := \left|
        \int \log \frac{K_\varepsilon*\mu(z)}{\pi(z)}\,\mu(z)\,dz
        - \int \log \frac{K_\varepsilon*\nu(z)}{\pi(z)}\,\nu(z)\,dz \right|.
\]
Thanks to \eqref{prop:lipschitz:1}, we have
\begin{equation}\label{eq: flow3 estination1}
(i) \leq L_F \left( \mathcal{W}_2(\mu,\nu) + \lvert x-y \rvert \right).
\end{equation}
For the term (ii), by triangle inequality we have
\begin{equation*}
    (ii)\leq \left| 
      \log \frac{K_\varepsilon*\mu(x)}{\pi(x)} 
      - \log \frac{K_\varepsilon*\nu(x)}{\pi(x)} \right|
     + \left| 
      \log \frac{K_\varepsilon*\nu(x)}{\pi(x)} 
      - \log \frac{K_\varepsilon*\nu(y)}{\pi(y)} \right|.
\end{equation*}
We bound these two term separately by \eqref{prop:lipschitz:2} and  \eqref{prop:lipschitz:4} and we get:
\begin{equation}\label{eq: flow3 estination2}
    (ii) \leq \frac{L_{K_\varepsilon}}{K_{\min,\varepsilon}}
      \mathcal{W}_1(\mu,\nu) +  \left(\frac{L_{K_\varepsilon}}{K_{\min,\varepsilon}} + \frac{L_\pi}{\pi_{\min}}\right)
      |x-y|.
\end{equation}

    Now let us check item $(iii)$.
    We write
\[
h_\mu(x)=\frac{1}{(K_\varepsilon*\mu)(x)}, \qquad h_\nu(x)=\frac{1}{(K_\varepsilon*\nu)(x)}.
\]
Decompose
\[
\begin{aligned}
&K_\varepsilon*\left(\frac{\mu}{K_\varepsilon*\mu}\right)(y)-K_\varepsilon*\left(\frac{\nu}{K_\varepsilon*\nu}\right)(y) \\
&=\int K_\varepsilon(y-x)h_\mu(x)\,d(\mu-\nu)(x)
+\int K_\varepsilon(y-x)\big(h_\mu(x)-h_\nu(x)\big)\,d\nu(x) \\
&=:T_1+T_2.
\end{aligned}
\]

For $T_1$, set $f_\mu(x)=K_\varepsilon(y-x)h_\mu(x)$. Using \eqref{prop:bounded:7} and \eqref{prop:lipschitz:9} and the product rule for Lipschitz functions,
\[
\operatorname{Lip}(f_\mu)\le \operatorname{Lip}(K_\varepsilon)\|h_\mu\|_\infty+\|K_\varepsilon\|_\infty\operatorname{Lip}(h_\mu)
\le \frac{L_{K_\varepsilon}}{K_{\min,\varepsilon}}+\frac{K_{\max,\varepsilon}L_{K_\varepsilon}}{K_{\min,\varepsilon}^2}.
\]
By the Kantorovich--Rubinstein theorem,
\[
|T_1|=\left|\int f_\mu\,d(\mu-\nu)\right|\le \operatorname{Lip}(f_\mu)\,\mathcal{W}_1(\mu,\nu)
\le \left(\frac{L_{K_\varepsilon}}{K_{\min,\varepsilon}}+\frac{K_{\max,\varepsilon}L_{K_\varepsilon}}{K_{\min,\varepsilon}^2}\right)\mathcal{W}_1(\mu,\nu).
\]

For $T_2$, first note that
\[
|h_\mu(x)-h_\nu(x)|
=\frac{|(K_\varepsilon*(\nu-\mu))(x)|}{(K_\varepsilon*\mu)(x)(K_\varepsilon*\nu)(x)}
\le \frac{1}{K_{\min,\varepsilon}^2}\sup_{z}\left| \int K_\varepsilon(z-w)\,d(\mu-\nu)(w)\right|.
\]
For each fixed $z$, the function $w\mapsto K_\varepsilon(z-w)$ is $L_{K_\varepsilon}$-Lipschitz. Applying the Kantorovich--Rubinstein theorem gives
\[
\sup_{z}\left| \int K_\varepsilon(z-w)\,d(\mu-\nu)(w)\right|\le L_{K_\varepsilon}\,\mathcal{W}_1(\mu,\nu).
\]
Hence $\|h_\mu-h_\nu\|_\infty\le \frac{L_{K_\varepsilon}}{K_{\min,\varepsilon}^2}\mathcal{W}_1(\mu,\nu)$, and therefore
\[
|T_2|\le \|K_\varepsilon\|_\infty\,\|h_\mu-h_\nu\|_\infty
\le \frac{K_{\max,\varepsilon}L_{K_\varepsilon}}{K_{\min,\varepsilon}^2}\,\mathcal{W}_1(\mu,\nu).
\]

Combining the bounds for $T_1$ and $T_2$ yields
\begin{equation}\label{eq: flow3 estimation3}
(iii) \le \left(\frac{L_{K_\varepsilon}}{K_{\min,\varepsilon}}+\frac{2K_{\max,\varepsilon}L_{K_\varepsilon}}{K_{\min,\varepsilon}^2}\right)\mathcal{W}_1(\mu,\nu).
\end{equation}
Now let us check item $(iv)$.

\begin{equation}\label{eq: flow3 estimation4}
    \begin{split}
        (iv) & \leq \left \lvert \int  \left(\log \frac{K_\varepsilon*\mu(x)}{\pi(x)} - \log \frac{K_\varepsilon*\nu(x)}{\pi(x)}\right) \mu(x) dx \right \rvert\\
        &+ \left \lvert\int \left ( \mu(x) -  \nu(x)\right)\log \frac{K_\varepsilon*\nu(x)}{\pi(x)} dx\right \rvert\\
        &  \leq  \int  \left \lvert\log K_\varepsilon*\mu(x) - \log K_\varepsilon*\nu(x)\right \rvert \mu(x) dx \\
        &+ \left \lvert\int \left (\mu(x) -  \nu(x)\right)\log \frac{K_\varepsilon*\nu(x)}{\pi(x)} dx\right \rvert\\
        &\leq \left(\frac{2L_{K_\varepsilon}}{K_{\min,\varepsilon}} + \frac{L_\pi}{\pi_{\min}}\right) \mathcal{W}_1(\mu,\nu),
    \end{split}
    \end{equation}
    where in the last inequality, we used \eqref{prop:lipschitz:2} and  \eqref{prop:lipschitz:5}.

Combining \eqref{eq: flow3 estination1}, \eqref{eq: flow3 estination2}, \eqref{eq: flow3 estimation3} and \eqref{eq: flow3 estimation4}, we could conclude that $a(m,x)$ is Lipschitz with constant
\begin{equation*}
    L^3_\varepsilon =L_F +\left(\frac{4L_{K_\varepsilon}}{K_{\min,\varepsilon}} + \frac{2L_\pi}{\pi_{\min}}\right)+\frac{2K_{\max,\varepsilon}L_{K_\varepsilon}}{K_{\min,\varepsilon}^2}.
\end{equation*}
\end{proof}

\subsection{\texorpdfstring{For Flow \eqref{eq: kernel_4}}{For flow (4)}}

Define the coefficient associated with the kernelized flow \eqref{eq: kernel_4} by
\[
a(\mu,x)
:= \frac{\delta F}{\delta \mu}(\mu,x)
+ \sigma\Big( \big(K_\varepsilon* \log\tfrac{K_\varepsilon*\mu}{\pi}\big)(x)
- \operatorname{KL}(K_\varepsilon * \mu \mid \pi) \Big).
\]
By Proposition~\ref{prop: a_lipschitz_bounded}, \(a\) satisfies both a boundedness estimate
and a Lipschitz estimate. In particular, under Assumptions \ref{hypo: F_full}, \ref{hypo: mollifier_kernel}, and \ref{hypo: pi_full},
there exist constants \(C^4_\varepsilon>0\) and \(L^4_\varepsilon>0\) such that, for all
\(\mu,\nu\in\mathcal{P}(\mathcal{X})\) and \(x,y\in\mathcal{X}\),
\[
|a(\mu,x)| \le C^4_\varepsilon,
\qquad
|a(\mu,x)-a(\nu,y)| \le L^4_\varepsilon\big(\mathcal{W}_2(\mu,\nu)+|x-y|\big).
\]

\begin{proof}
\textit{Boundedness.}
By \eqref{prop:bounded:2} and $\int_{\mathcal{X}} K_\varepsilon=1$,
\begin{align*}
\left| \big(K_\varepsilon * \log \tfrac{K_\varepsilon* \mu}{\pi}\big)(x) \right|
&= \left| \int K_\varepsilon(x-y)\,\log \tfrac{(K_\varepsilon* \mu)(y)}{\pi(y)}\,dy \right| \\
&\le \int K_\varepsilon(x-y)\,\left| \log \tfrac{(K_\varepsilon* \mu)(y)}{\pi(y)} \right| dy \\
&\le \max\!\left( \log\frac{\pi_{\max}}{K_{\min,\varepsilon}},\ \log\frac{K_{\max,\varepsilon}}{\pi_{\min}} \right).
\end{align*}
By the same reasoning, $\operatorname{KL}(K_\varepsilon* \mu | \pi)$ is bounded by the same constant. Hence we may take
\[
C_\varepsilon^4 := C_F + \sigma\,\max\!\left( \log\frac{\pi_{\max}}{K_{\min,\varepsilon}},\ \log\frac{K_{\max,\varepsilon}}{\pi_{\min}} \right).
\]

We now turn to the proof of the Lipschitz condition. By the triangle inequality,
\begin{align*}
|a(\mu,x)-a(\nu,y)|
&\le (i)+\sigma(ii)+\sigma(iii),
\end{align*}
where
\[
(i):=\Big|\tfrac{\delta F}{\delta \mu}(\mu,x)-\tfrac{\delta F}{\delta \mu}(\nu,y)\Big|,
\quad
(ii):=\Big|\big(K_\varepsilon* \log \tfrac{K_\varepsilon*\mu}{\pi}\big)(x)
-\big(K_\varepsilon* \log \tfrac{K_\varepsilon*\nu}{\pi}\big)(y)\Big|,
\]
\[
(iii):=\Big|\operatorname{KL}(K_\varepsilon * \mu | \pi)-\operatorname{KL}(K_\varepsilon * \nu | \pi)\Big|.
\]
For $(i)$, Assumption~\ref{hypo: F_full} yields
\begin{equation}\label{eq: flow4 estimation1}
(i)\le L_F\big(\mathcal{W}_2(\mu,\nu)+|x-y|\big).
\end{equation}
For $(ii)$, split into a measure part and a space part:
\begin{equation}\label{eq: flow4 estimation2}
\begin{split}
(ii)\ &\le\ \Big|\big(K_\varepsilon* \log \tfrac{K_\varepsilon*\mu}{\pi}\big)(x)
-\big(K_\varepsilon* \log \tfrac{K_\varepsilon*\nu}{\pi}\big)(x)\Big| \\
& + \Big|\big(K_\varepsilon* \log \tfrac{K_\varepsilon*\nu}{\pi}\big)(x)
-\big(K_\varepsilon* \log \tfrac{K_\varepsilon*\nu}{\pi}\big)(y)\Big|\\
& \leq \int K_\varepsilon(x-y) \left | \log \frac{K_\varepsilon * \mu}{\pi}(y) -\log \frac{K_\varepsilon * \nu}{\pi}(y)\right| dy\\
& + \int K_\varepsilon(z) \left| \log \frac{K_\varepsilon * \nu}{\pi}(x-z) -\log \frac{K_\varepsilon * \nu}{\pi}(y-z)\right|dx\\
& \leq \left(\frac{L_{K_\varepsilon}}{K_{\min,\varepsilon}}+\frac{L_\pi}{\pi_{\min}}\right)\big(\mathcal{W}_1(\mu,\nu)+|x-y|\big),
\end{split}
\end{equation}
where the last inequality is from \eqref{prop:lipschitz:2} and \eqref{prop:lipschitz:5}.

For $(iii)$, 

\begin{equation}\label{eq: flow4 estimation3}
    (iii) \leq \left(\frac{2L_{K_\varepsilon}}{K_{\min,\varepsilon}} + \frac{L_\pi}{\pi_{\min}}\right) \mathcal{W}_1(\mu,\nu),
\end{equation}
because of \eqref{prop:lipschitz:8}.

Combining \eqref{eq: flow4 estimation1}, \eqref{eq: flow4 estimation2} and \eqref{eq: flow4 estimation3} we could choose
\[
L^4_\varepsilon = L_F+\sigma\left(\frac{3L_{K_\varepsilon}}{K_{\min,\varepsilon}}+\frac{2L_\pi}{\pi_{\min}}\right).
\]
\end{proof}

\subsection{Boundness and Lipschitz Condition for Kernelized Fisher-Rao Gradient Flow Regularized by Chi Square Divergence}\label{subsec: chisq}

\begin{proof}
Using the uniform bound on $\frac{\delta F}{\delta \mu}$ and $\pi\ge \pi_{\min}$,
\[
\left|K_\varepsilon\!\ast\!\left(\frac{K_\varepsilon\!\ast\!\mu}{\pi}\right)(x)\right|
= \left|\int K_\varepsilon(x-z)\,\frac{(K_\varepsilon\!\ast\!\mu)(z)}{\pi(z)}\,dz\right|
\le \frac{K_{\max,\varepsilon}}{\pi_{\min}},
\]
and the same bound holds for $\int K_\varepsilon\!\ast\!\left(\frac{K_\varepsilon\!\ast\!\mu}{\pi}\right)\,d\mu$.
Hence
\[
|a(\mu,x)|
\le C_F + 4\sigma\,\frac{K_{\max,\varepsilon}}{\pi_{\min}}
=: C_\varepsilon^{5}.
\]

We now analyze the Lipschitz condition. By the triangle inequality,
\begin{equation*}
|a(\mu,x)-a(\nu,y)| \leq (i) + 2\sigma (ii) + 2\sigma (iii),
\end{equation*}
where
\[
(i):=\left|\frac{\delta F}{\delta \mu}(\mu,x)-\frac{\delta F}{\delta \mu}(\nu,y)\right|,
\]
\[
(ii):=\left|K_\varepsilon*\left(\frac{K_\varepsilon*\mu}{\pi}\right)(x)-K_\varepsilon*\left(\frac{K_\varepsilon*\nu}{\pi}\right)(y)\right|,
\]
\[
(iii):=\left|\int K_\varepsilon*\left(\frac{K_\varepsilon*\mu}{\pi}\right)\,d\mu - \int K_\varepsilon*\left(\frac{K_\varepsilon*\nu}{\pi}\right)\,d\nu\right|.
\]
For $(i)$,
\[
(i)\le L_F\big(\mathcal{W}_2(\mu,\nu)+|x-y|\big).
\]
For $(ii)$, split
\[
\left|K_\varepsilon\!\ast\!\left(\frac{K_\varepsilon\!\ast\!\mu}{\pi}\right)(x)-K_\varepsilon\!\ast\!\left(\frac{K_\varepsilon\!\ast\!\nu}{\pi}\right)(x)\right|
= \left|\int \phi_x(z)\,d(\mu-\nu)(z)\right|,
\]
with
\[
\phi_x(z):=\int \frac{K_\varepsilon(x-w)}{\pi(w)}\,K_\varepsilon(z-w)\,dw.
\]
Then $\operatorname{Lip}(\phi_x)\le \frac{L_{K_\varepsilon}}{\pi_{\min}}$, hence by the Kantorovich--Rubinstein theorem,
\[
\left|K_\varepsilon\!\ast\!\left(\frac{K_\varepsilon\!\ast\!\mu}{\pi}\right)(x)-K_\varepsilon\!\ast\!\left(\frac{K_\varepsilon\!\ast\!\nu}{\pi}\right)(x)\right|
\le \frac{L_{K_\varepsilon}}{\pi_{\min}}\,\mathcal{W}_1(\mu,\nu).
\]
Moreover,
\[
\left|K_\varepsilon\!\ast\!\left(\frac{K_\varepsilon\!\ast\!\nu}{\pi}\right)(x)-K_\varepsilon\!\ast\!\left(\frac{K_\varepsilon\!\ast\!\nu}{\pi}\right)(y)\right|
\le \frac{L_{K_\varepsilon}}{\pi_{\min}}\,|x-y|.
\]
Thus
\[
(ii)\le \frac{L_{K_\varepsilon}}{\pi_{\min}}\big(\mathcal{W}_1(\mu,\nu)+|x-y|\big).
\]
For $(iii)$, set $f_\mu(z)=K_\varepsilon\!\ast\!\left(\frac{K_\varepsilon\!\ast\!\mu}{\pi}\right)(z)$ and write
\[
(iii)\le \left|\int K_\varepsilon* \left( \frac{K_\varepsilon*\mu}{\pi} \right)d(\mu-\nu)\right| + \int \left| K_\varepsilon* \left( \frac{K_\varepsilon*\mu}{\pi} \right)- K_\varepsilon* \left( \frac{K_\varepsilon*\nu}{\pi} \right) \right|\,d\nu.
\]
Since $\operatorname{Lip}( K_\varepsilon* \left( \frac{K_\varepsilon*\mu}{\pi} \right))\le \operatorname{Lip}( \frac{K_\varepsilon*\mu}{\pi} ) \frac{L_{K_\varepsilon}}{\pi_{\min}}$, the first item is bounded by $\frac{L_{K_\varepsilon}}{\pi_{\min}}\,\mathcal{W}_1(\mu,\nu)$.
Furthermore,
\[
\left |K_\varepsilon* \left( \frac{K_\varepsilon*\mu}{\pi} \right)(x)-K_\varepsilon* \left( \frac{K_\varepsilon*\nu}{\pi}(x) \right) \right|
\le \left\|\frac{K_\varepsilon\!\ast\!(\mu-\nu)}{\pi}\right\|_\infty
\le \frac{L_{K_\varepsilon}}{\pi_{\min}}\,\mathcal{W}_1(\mu,\nu),
\]
hence
\[
(iii)\le \frac{2L_{K_\varepsilon}}{\pi_{\min}}\,\mathcal{W}_1(\mu,\nu).
\]
Using $\mathcal{W}_1(\mu,\nu)\le \mathcal{W}_2(\mu,\nu)$ and collecting the bounds gives
\[
|a(\mu,x)-a(\nu,y)|
\le \left(L_F + 6\sigma\,\frac{L_{K_\varepsilon}}{\pi_{\min}}\right)
\big(\mathcal{W}_2(\mu,\nu)+|x-y|\big).
\]
Therefore the Lipschitz estimate holds with
\[
L_\varepsilon^{5} := L_F + 6\sigma\,\frac{L_{K_\varepsilon}}{\pi_{\min}}.
\]
\end{proof}

\end{document}